\documentclass[12pt,a4paper,reqno]{amsart}

\usepackage{amsthm}
\usepackage{epsfig}
\usepackage{alltt}
 \usepackage[active]{srcltx}
\usepackage{newlfont}
\usepackage{amsmath}
\usepackage{amssymb}
\usepackage{amsfonts}
\usepackage{amscd}
\usepackage{enumitem}
\setenumerate[1]{label=(\alph*)}
\setenumerate[2]{label=(\roman*)}
\usepackage{fancybox}

\usepackage{verbatim}
\usepackage{xspace} 
\usepackage{changebar}
\usepackage[OT2,T1]{fontenc}
\usepackage{hyperref}  
\usepackage{mathrsfs}
\usepackage{ulem}
\input{xy}
\xyoption{all}

\newtheorem{Satz}{Satz}[section]

\newtheorem{Proposition}[Satz]{Proposition}
\newtheorem{Proposition*}[Satz]{Proposition*}
\newtheorem{corollary}[Satz]{Corollary}
\newtheorem{Korollar*}[Satz]{Korollar*}
\newtheorem{Lemma}[Satz]{Lemma}
\newtheorem{Lemma*}[Satz]{Lemma*}
\newtheorem{theorem}[Satz]{Theorem}

\theoremstyle{definition}
\newtheorem{Definition}[Satz]{Definition}

\theoremstyle{remark}
\newtheorem{Bemerkungl}[Satz]{}

\newtheorem*{App1*}{Locally closed embeddings}
\newtheorem*{App2*}{Homotopy invariance of sheaf cohomology}
\newtheorem*{App3*}{Equivariant derived categories}
\newtheorem*{Acknowledgment*}{Acknowledgment}

\newtheorem{Bemerkung*}[Satz]{Bemerkung*}

\newsavebox{\kD}
\savebox{\kD}(7,7){
\begin{picture}(7,7)
 \put(7,0){\line(-1,0){7}}
\put(7,0){\line(0,1){7}}
\end{picture}}
\newcommand{\kart}{\ar @{} [dr] |{\begin{picture}(7,7)
 \put(0,0){\usebox{\kD}}
\end{picture}}}

\newsavebox{\ckD}
\savebox{\ckD}(7,7){
\begin{picture}(7,7)
 \put(0,7){\line(0,-1){7}}
\put(0,7){\line(1,0){7}}
\end{picture}}

\newcommand{\ol}[1]{\overline{#1}}

\newcommand{\Sh}{\operatorname{Sh}}

\newcommand{\Ab}{\operatorname{Ab}}
\newcommand{\Hom}{\operatorname{Hom}}
\newcommand{\supp}{\operatorname{supp}}
\newcommand{\sheafHom}{{\mkern3mu\mathcal{H}{om}\mkern3mu}}
\newcommand{\xra}[1]{\xrightarrow{#1}}

\newcommand{\sira}{\xra{\sim}}

\newcommand{\sila}{\overset{\sim}{\leftarrow}}
\newcommand{\opp}{{\operatorname{op}}}

\renewcommand{\amalg}{\sqcup}

\newcommand{\ra}{\rightarrow}

\newcommand{\sra}{\twoheadrightarrow}
\newcommand{\hra}{\hookrightarrow}

\newcommand{\Bl}[1]{{\mathbb{#1}}}

\newcommand{\DZ}{\Bl{Z}}
\newcommand{\DN}{\Bl{N}}

\newcommand{\defind}[1]{{\bf #1}\index{#1}}

\newcommand{\op}{\operatorname}

\newcommand{\co}{\subset}

\newcommand{\pdef}{\mathrel{\mathop:}=}

\begin{document}
\title[Locally proper maps]{Proper base change for separated
  locally proper maps}

\author{Olaf M.\ Schn\"urer
  \and Wolfgang Soergel}
\thanks{Olaf Schn\"urer was supported by the SPP 1388 and the SFB/TR 45 of the DFG}
\thanks{Wolfgang Soergel was supported by the SPP 1388 of the DFG}
\address{Olaf Schn\"urer, 
  Mathematisches Institut,
  Rheinische Friedrich-Wil\-helms-Universit{\"a}t Bonn,
  Endenicher Allee 60,
  53115 Bonn,
  Germany}
\email{olaf.schnuerer@math.uni-bonn.de}
\address{Wolfgang Soergel, Mathematisches Institut,
  Albert-Ludwigs-Uni\-ver\-si\-t\"at Freiburg, Eckerstra\ss{}e 1,
  79104 Freiburg im Breisgau, Germany}
\email{wolfgang.soergel@math.uni-freiburg.de}

\subjclass[2010]{}

\begin{abstract}
  We introduce and study the notion of a locally proper map
  between 
  topological spaces.
  We show that fundamental constructions of sheaf theory,
  more precisely proper base change, projection formula,
  and Verdier duality, can be extended from 
  continuous maps between locally compact Hausdorff spaces
  to separated locally proper maps between
  arbitrary topological spaces.
  %
\end{abstract}

\maketitle 
\setcounter{tocdepth}{1}
\tableofcontents

\section{Introduction}
\label{sec:introduction}

The proper
direct
image functor $f_!$ and its right derived functor ${\op{R}}f_!$  
are defined for any  continuous map  $f \colon Y\ra X$ 
of locally compact Hausdorff spaces, see \cite{KS,spaltenstein,verdier-dualite-bourbaki}.
Instead of 
$f_!$ and  ${\op{R}}f_!$  
we will use the notation $f_{(!)}$ and
$f_!$ for these functors. They are embedded into
a whole collection of formulas known as
the six-functor-formalism of Grothendieck. 
Under the assumption that the proper direct image functor
$f_{(!)}$ has finite cohomological dimension, Verdier proved that
its derived functor $f_!$ admits a right adjoint $f^!$.

In this article we introduce in \ref{LEAm} the notion of a
locally proper map between topological spaces and show that the
above results even hold for arbitrary topological spaces if all
maps whose proper direct image functors are involved are locally
proper and separated.  Every continuous map between locally
compact Hausdorff spaces is separated and locally proper by
\ref{c:loc-cpt-Hd-spaces-sep-lp}, and in this case our results
specialize to the classical theory.

The basic properties of locally proper maps are established in
section~\ref{sec:locally-proper-maps}. 
In particular a map from a topological space  to  a
one point space is separated and locally proper if and only if our space 
is locally compact Hausdorff. Since  the properties of
being 
separated and locally proper are moreover stable under base change 
by \ref{sep-BC} and  \ref{l:BC-localprop}, a separated
locally proper map can be interpreted as a continuous family of
locally compact Hausdorff spaces. 

Our main sheaf theoretic results concerning separated locally
proper maps are the 
proper base change
\ref{t:BaWe}, the projection formula
\ref{t:ProFor}, the derived proper base change
\ref{t:DeBaW}, the derived projection formula
\ref{t:ANNO}, and Verdier duality
\ref{t:VdD-bounded}.
Theorem~\ref{t:comp-bc-proj} 
explains  how all the previous
results work in the setting of unbounded derived categories. 

In the following we list three applications of our extension of
the classical theory described above. 

\begin{App1*}
  An embedding $i \colon Z \hra X$ of topological spaces is
  always separated, and by \ref{l:embedding-locprop} it is
  locally proper if and only if $i(Z)$ is a locally closed subset
  of $X$. In this case the functor $i_{(!)}=i_!$ is the extension
  by zero and the functor $\mathcal{H}^p i^!$ maps a sheaf to its
  $p$-th local cohomology sheaf.
\end{App1*}

\begin{App2*}
  Homotopy invariance of sheaf cohomology can be deduced, as
  explained in \cite[2.7]{KS}, from the fact that given a space
  $X$ and a complex of sheaves $\mathcal F$ on $X$, for the
  projection $\pi\colon X\times [0,1]\ra X$ the unit of the adjunction
  of derived direct and inverse image is an isomorphism $\mathcal
  F\sira \pi_*\pi^*\mathcal F$. Now for an arbitrary topological
  space $X$ the map $\pi$ is always proper and separated, so
  proper base change in the generality of this article allows to
  reduce the claim to the case when $X$ is a one point space,
  which is treated in the literature, say in \cite{KS}.
\end{App2*}

\begin{App3*}
  Fix a topological group $G$ and let $X$ be a $G$-space, i.\,e.\
  a topological space with a continuous $G$-action $G \times X
  \ra X.$ Let $\op{E}G$ be a contractible space with a
  topologically free $G$-action, e.\,g.\ the Milnor construction.
  Then by homotopy invariance of sheaf cohomology the map
  $\op{E}G \ra \{\op{pt}\}$ is acyclic, and so is $p\colon
  \op{E}G \times X \ra X$.  Let $q\colon \op{E}G \times X \ra
  \op{E}G \times_G X$ be the quotient map.
  Then the bounded below equivariant derived category can by
  \cite[2.9.4]{BeLu} be described as the full triangulated
  subcategory of $D^+ (\op{E}G \times_G X)$ given by
  \begin{equation*}
    D^+_G (X) := \{\mathcal{F} \in D^+ (\op{E}G \times_G X)\mid
    \exists \mathcal{G} \in D^+ (X) \text{ such that }
    q^\ast \mathcal{F} \cong p^\ast \mathcal{G}\}.
  \end{equation*}
  Now let $f \colon X \rightarrow Y$ be a $G$-equivariant
  continuous map of locally compact Hausdorff $G$-spaces or even
  a $G$-equivariant separated locally proper map of arbitrary
  $G$-spaces.  We obtain a diagram
  \begin{displaymath}
    \xymatrix{
      \op{E}G \times_G X \ar[d]^-{\bar f} &\ar[l]\op{E}G \times X 
      \ar[d]^-{\hat f} \ar[r] & X \ar[d]^f\\ 
      \op{E}G \times_G Y & \ar[l] \op{E}G \times Y \ar[r] & Y
    }
  \end{displaymath}
  whose vertical maps are induced by $f$. Both squares are
  cartesian so that $\hat{f}$ and $\bar{f}$ are separated locally
  proper by \ref{l:BC-localprop}, \ref{locprop-localtarget},
  \ref{sep-local-target}.  Note however that, given a locally
  compact Hausdorff space $Z$, the spaces $\op{E}G \times Z$ and
  $\op{E}G \times_G Z$ are, in general, not locally compact
  Hausdorff.  Derived proper base change~\ref{t:DeBaW} in the
  generality of this article nevertheless shows that $\bar{f}_!$
  induces a functor $f_! \colon D^+_G (X) \rightarrow D^+_G (Y)$
  which generalizes the proper direct image functor of
  \cite{BeLu}.
\end{App3*}

In the hope that the reader is now sufficiently motivated, let us just add that
from our point of view the main 
new ingredients are the definition of a locally
proper map and the proof of the underived form of proper base change.
Once this is done, we just have to document that the standard arguments
work in this generality as well.




\begin{Acknowledgment*}
  We thank Hanno Becker for useful discussions.
\end{Acknowledgment*}

\section{Locally proper maps}
\label{sec:locally-proper-maps}
\begin{Bemerkungl}
  \label{d:cpt-locally-cpt}
  A topological space  is called \textbf{compact} if every
  open 
  covering has a finite subcovering.  A topological space  is called
 \textbf{locally
    compact} if 
  every neighborhood of any point contains a compact neighborhood
  of this point. We do not require the
  Hausdorff property in either case.
\end{Bemerkungl}
\begin{Bemerkungl}
  The definitions of proper and separated maps and their
  basic properties are recalled in
  Section~\ref{sec:reminders}. 
\end{Bemerkungl}

\begin{Definition}
  \label{LEAm} 
  A map
  $f \colon Y \ra X$ of topological spaces is
called  \textbf{locally proper}
  if 
  it is continuous and if 
  given
  any 
  point $y \in Y$ and any neighborhood $V$ of $y$
  there are neighborhoods $A\subset V$ of $y$ and $U$
  of $f(y)$ such that $f(A) \subset U$ and the induced map
  $f\colon  A \ra U$ is proper.
\end{Definition}

\begin{Bemerkungl}
  The constant map from a topological space $Y$ to a 
  space consisting of a single point is locally proper if and
  only if $Y$ is locally compact.
\end{Bemerkungl}

\begin{Lemma}
  \label{l:embedding-locprop}
  An embedding $i\colon Y\hra X$ of topological spaces is locally
  proper if and only if 
  $i(Y)$ is a locally closed subset of $X.$
\end{Lemma}

\begin{proof}
  We can assume that $i$ is the inclusion of a subset
  $Y$ of $X$ with its induced topology. Recall from \ref{VSU}
  that an embedding is proper if and only if
  it is closed. 

  Assume that $i$ is locally proper. 
  Given $y \in Y$ there
  are a neighborhood $A$ of $y$ in $Y$ and a neighborhood $U$ of
  $y$ in $X$ with $A \subset U$ such that the inclusion $A
  \subset U$ is 
  closed.  
  By replacing $U$ with a smaller open neighborhood of
  $y$ and $A$ by its intersection with this neighborhood we can
  assume that $U$ is open in $X.$ 
  Since $A$ is a neighborhood of
  $y$ there is an open subset $W \subset X$ such that $y \in W
  \cap Y \subset A.$ Then the inclusion $W \cap A \subset W \cap
  U$ is closed and
  $W \cap A = W \cap Y=(W \cap U) \cap Y.$ This shows that $Y$ is
  a locally closed subset of $X.$

  The converse implication is obvious.
\end{proof}

\begin{Bemerkungl}
  \label{rem:comp-locprop}
  Every composition of locally proper maps is locally proper.
  This follows easily using that a composition of proper maps is
  proper by \ref{comp-proper} and that any base change of a proper
  map is proper by \ref{rem:prop-BC}.
\end{Bemerkungl}

\begin{Lemma}[\textbf{Local properness and base change}]
  \label{l:BC-localprop}
  Local properness is stable under base change.
   More precisely, let 
  \begin{equation*}
    \xymatrix{\kart 
      W \ar[r]^{q}\ar[d]_{g} & Y\ar[d]^{f}\\
      Z \ar[r]^{p} & X 
    }  
  \end{equation*}
  be a cartesian diagram of topological spaces and continuous maps.
  If $f$ is
  locally proper so is $g$.
\end{Lemma}

\begin{proof}
  Let $w\in W$ together with a neighborhood $S \subset W$ be
  given.
  Let us identify $W =Z \times_X Y.$ Then $w=(g(w),q(w)).$ We
  find neighborhoods $V\subset Y$ of $q(w)$ and 
  $U\subset Z$ of $g(w)$ such that $U \times_X V \subset S.$
  By replacing $S$ by $U \times_X V$ we can assume that the
  diagram 
  $$\xymatrix{\kart
    S \ar[r]^{q}\ar[d]_{g} & V\ar[d]^{f}\\
    U \ar[r]^{p} & X }$$ is cartesian.  We find neighborhoods
  $A\subset V$ of $q(w)$ and $N \subset X$ of $f(q(w))$ such that
  $f\colon A \ra N$ is proper. The same is true for the map
  $g\colon q^{-1}(A)\ra p^{-1}(N)$ obtained by base change, by
  \ref{rem:prop-BC}.
\end{proof}

\begin{Bemerkungl}
  \label{locprop-localsource}
  The property of being locally proper is local on the
  source.
  Namely, let 
  $f \colon  Y \ra X$ be a map and let $\mathcal{V}$ be an
  open covering of $Y$. Then $f$ is locally 
  proper if and only if its restrictions $f|_V\colon V \ra
  Y$ are locally proper for all $V \in \mathcal{V}$, as follows easily
    using
 \ref{l:embedding-locprop}.
\end{Bemerkungl}

\begin{Bemerkungl}
  \label{locprop-localtarget}
  The property of being locally proper is local on the
  target.
  Namely, let 
  $f \colon  Y \ra X$ be a map and let $\mathcal{U}$ be an
  open covering of $X$. Then $f$ is locally 
  proper if and only if the induced maps $f^{-1}(U) \ra U$ are
  locally proper for all 
  $U \in \mathcal{U}$, as follows easily
    using
 \ref{l:BC-localprop}.  
\end{Bemerkungl}

\begin{Lemma}
  \label{l:locprop-to-sep}
  Let $g \colon Z \ra Y$ and $f \colon Y \ra X$ be
  continuous
  maps of 
  topological spaces. Assume that $f \circ g$ is locally proper
  and that $f$ is separated.
  Then $g$ is
  locally proper.
\end{Lemma}

\begin{proof}
  Let $V \subset Z$ be a given neighborhood of a point $z \in Z.$
  Since $f \circ g$ is locally proper, there are neighborhoods
  $A \subset V$ of $z$ and $U$ of $f(g(z))$ such that $f \circ g$
  induces a proper map $A \ra U.$
  This map factors as
  \begin{equation*}
    A \xrightarrow{g'} f^{-1}(U) \xrightarrow{f'} U
  \end{equation*}
  where $g'$ and $f'$ are induced by $g$ and $f$ respectively.
  Since $f'$ is separated by \ref{sep-BC},
  Lemma~\ref{l:SaA} shows that 
  $g'\colon A \ra f^{-1}(U)$ is proper.
  This shows that $g$ is locally proper.
\end{proof}

\begin{corollary}
  \label{c:loc-cpt-Hd-spaces-sep-lp}
  Every continuous map $g\colon Z \ra Y$ 
  from a locally compact
  Hausdorff space $Z$ to a Hausdorff space $Y$ is separated and
  locally proper. 
\end{corollary}

\begin{proof}
  Since $Z$ is locally compact and $Y$ is Hausdorff, $g$ is
  locally proper by 
  \ref{l:locprop-to-sep}. Since $Z$ is Hausdorff, 
  $g$ is separated.
\end{proof}

\begin{Proposition}
  \label{p:sele}
  Every proper and separated map is locally proper.
\end{Proposition}

\begin{Bemerkungl}
  This generalizes the fact that every compact Hausdorff space is
  locally compact.
\end{Bemerkungl}

\begin{proof}
  Let $f \colon Y \ra X$ be a proper and separated map.
  Let $y \in Y$ together with an open neighborhood $V \co Y$ be
  given. Put $x \pdef
  f(y)$ and consider $Z = Y \setminus V$. For every point $z \in f^{-1} (x) \cap Z$ there is an open
  neighborhood $W_z \co Y$ of $z$ and an open neighborhood $B_z
  \co V$ of $y$ with $W_z \cap B_z = \emptyset$. Finitely many of
  these $W_z$ cover the compact set $f^{-1} (x) \cap
  Z$. Let $W$ be the union of these sets and $B$ the
  intersection of the corresponding sets $B_z$. Then $y \in B$,
  $B$ is open in $V$, $(f^{-1} (x) 
  \cap Z) \subset W$, $W$ is open in $Y$, and $B \cap W =
  \emptyset$. Now 
  we find an open neighborhood $U \co X$ of $x$ with $(f^{-1}
  (U) \cap Z) \subset W$, we may for example take the complement of
  the closed subset $f(Z\ \cap (Y\setminus W))$ of $X$.
  Then
  the closure $A$
  of $B \cap f^{-1} (U)$ in $f^{-1} (U)$ has empty intersection
  with $f^{-1} (U)
  \cap Z$ and hence is a neighborhood of $y$ contained in
  $V$ such that the induced 
  map $f\colon  A \rightarrow U$ is proper.
\end{proof}

\section{Proper direct image}
\label{sec:proper-direct-image}

\begin{Bemerkungl}
  By a sheaf we mean a sheaf of abelian groups.
  The constant sheaf with stalk $\DZ$ on a topological space
  $X$ is denoted $\DZ_X.$
  Given a sheaf $\mathcal{F}$ on $X$ we abbreviate as usual
  $\Gamma\mathcal{F}= \Gamma(\mathcal{F})=
  \Gamma(X;\mathcal{F})= \mathcal{F}(X)$ its 
set of global sections 
and write
  $\mathcal{F}(Z)=\Gamma(Z; \mathcal{F})= \Gamma(Z;
  \mathcal{F}|_Z)$ if $Z \subset X$ is an arbitrary subset of
  $X.$

  We write $\Sh(X)$ 
   for the abelian category of sheaves on $X$.
  Given a complex $\mathcal{F}$ in $\Sh(X)$
  we denote its $p$-th cohomology sheaf by
  $\mathcal{H}^p(\mathcal{F})$.
  We write $C(X)$ 
  for the abelian category of
  complexes 
  in $\Sh(X)$,
  denote the corresponding homotopy category by
  $K(X)$ 
  and the corresponding derived category
  by $D(X)$. Let $D^+(X) \subset D(X)$ be the full triangulated
  subcategory consisting of complexes with bounded below
  cohomology sheaves. 

  Given a continuous map $f$ of topological spaces
  we use the notation
  $f^{(\ast)}$ and $f_{(\ast)}$ for the inverse and direct image
  functors of sheaves, in contrast to the usual notation
  $f^{-1}$ and $f_\ast$ in the literature, and
  denote their derived functors by
  $f^\ast$ and $f_\ast$, in contrast to the usual notation
  $\mathrm Lf^\ast$ and $\mathrm Rf_\ast$.

  Given a category $\mathcal{C}$ and
  objects $A,$ $B\in\mathcal{C}$ we denote the set of morphisms by
  $\mathcal{C}(A,B)$ or $\Hom_\mathcal{C}(A,B).$
 We
  write $\Sh_X(A,B)$ instead of $(\Sh(X))(A,B).$
\end{Bemerkungl}

\begin{Definition}
  Let $f\colon Y\ra X$ be a continuous map of topological
  spaces. The  {\bf
    proper direct image}\label{EDBi} 
  $f_{(!)} \colon  \Sh(Y) \ra \Sh(X)$
  is defined by
  \begin{equation*}
    (f_{(!)}\mathcal{F}) (U) \pdef \{ s \in \mathcal{F} 
    (f^{-1}(U)) \mid \text{$f\colon (\supp s) \ra U$ is proper}
    \}.  
  \end{equation*}
  Since by \ref{prop-localtarget} properness is local on the target,
  $f_{(!)}  \mathcal{F} \subset f_{(\ast)} \mathcal{F}$ is a
  subsheaf of sets and by \ref{VUAa} even a subsheaf of
  abelian groups.
  Obviously, $f_{(!)}$ is a left exact
  functor. For proper $f$ we have 
  $f_{(!)}= f_{(\ast)}$ by \ref{VSU}.  
  Given a sheaf
  $\mathcal{F}$ on $Y$ we denote by
  \begin{equation*}
    \Gamma_!\mathcal{F}=\Gamma_!(\mathcal{F})= 
    \Gamma_!(Y;\mathcal{F})=(c_{(!)}\mathcal{F})(X)
  \end{equation*}
  the abelian group of sections with compact support where $c$ is
  the unique map from $Y$ to a space consisting of a single
  point.  If $Z \subset Y$ is any subset we abbreviate
  $\Gamma_!(Z;\mathcal{F})=\Gamma_!(Z; \mathcal{F}|_Z).$
\end{Definition}

\begin{Bemerkungl}
  Usually the functor $f_{(!)}$ is studied for continuous
  maps of locally compact Hausdorff spaces.
  Our aim is to show that $f_{(!)}$ behaves well in
  a more general setting, namely for separated locally
  proper maps $f$. 
 \end{Bemerkungl}

\begin{Bemerkungl}
  \label{rem:shriek-embedding}
  Let $i\colon Y \hra X$ be an embedding of a locally closed
  subset. Then 
  $i_{(!)}\mathcal{F}$ is the extension of $\mathcal{F}\in \Sh(Y)$
  by zero which is denoted by $i_!\mathcal{F}$ in 
  \cite[Exp.~I]{SGA-2-new}
  and by $\mathcal{F}^X$ in \cite[Thm.~2.9.2]{godement}.
  The functor $i_{(!)} \colon \Sh(Y) \ra \Sh(X)$ is exact and has a
  right adjoint functor $i^{(!)}$ which is denoted $i^!$ in 
  \cite[Exp.~I]{SGA-2-new}.
  If $Y$ is closed in $X$ we have $i_{(!)}=i_{(\ast)}$.
  If $Y$ is open in $X$ we have $i^{(!)}=i^{(\ast)}$.
\end{Bemerkungl}

\begin{Bemerkungl}
  \label{rem:supp-direct-image}
  Let $f \colon Y \ra X$ be a continuous map and $\mathcal{F}$ a
  sheaf on $Y.$ 
  A global section $s \in (f_{(\ast)}\mathcal{F})(X)$
  is the same thing as
  a global section
  $s' \in \mathcal{F}(Y).$
  The supports of these sections are related by
  $\overline{f(\supp s')}=\supp s.$ 
\end{Bemerkungl}

\begin{theorem}[{\bf Iterated proper images}]
  \label{t:DBE}
  Given continuous maps $g \colon Z \ra Y$ and $f\colon Y \ra X$
  with $f$ separated, the usual identity
  $f_{(\ast) }\circ g_{(\ast)} =
  (f \circ g)_{(\ast)}$ 
  of functors $\Sh(Z) \ra \Sh(X)$ restricts to an identity
  \begin{equation*}
    f_{(!)} \circ g_{(!)} = (f \circ g)_{(!)}.
  \end{equation*}
\end{theorem}

\begin{proof}
  Let $\mathcal{F}$ be a sheaf on $Z$ and $U\co X$ an open
  subset.
  Let $s \in (f_{(\ast)}(g_{(\ast)}\mathcal{F}))(U)$ be a section.
  We can interpret $s$ also as a
  section $s' \in (g_{(\ast)}\mathcal{F})(f^{-1}(U))$
  or as a section $s'' \in \mathcal{F}(g^{-1}(f^{-1}(U))),$ and also as
  a section $t \in ((f \circ g)_{(\ast)}\mathcal{F})(U).$

  Assume that $s$ is in $(f_{(!)}(g_{(!)}\mathcal{F}))(U).$ 
  Equivalently, this means that $s'$ is in $(g_{(!)}
  \mathcal{F})(f^{-1}(U))$ with $f\colon (\supp s') \ra U$ proper
  or, equivalently, that 
  $s''$ is in $\mathcal{F} (g^{-1}(f^{-1}(U)))$ with
  $g \colon (\supp s'') \ra f^{-1}(U)$ proper and
  $f \colon 
  g(\supp s'') \ra U$ proper. Here we use
  \ref{rem:supp-direct-image} and the fact that $g(\supp s'')$ is
  closed in $f^{-1}(U)$.
  Then the induced map 
  $g \colon (\supp s'') \ra g(\supp s'')$
  is proper by 
  \ref{l:SaA} and so is the composition
  $f \circ g \colon (\supp s'') \ra U$ by
  \ref{comp-proper}.
  This means that $t \in ((f \circ g)_{(!)}\mathcal{F})(U).$
  We have not yet used that $f$ is separated.

  Conversely, assume that
  $t \in ((f \circ g)_{(!)}\mathcal{F})(U).$
  This means that
  $f\circ g\colon (\supp s'') \ra U$ is proper.
  Since $f$ is separated, $g\colon (\supp s'') \ra f^{-1} (U)$ is
  proper  
  by \ref{l:SaA}, so 
  $s'$ is in $(g_{(!)} \mathcal{F})(f^{-1}(U))$
  and 
  $(\supp s') = g(\supp s'').$
  This and \ref{EWEi} imply that
  $f\colon (\supp s') \ra U$ is proper.
  Hence  $s \in (f_{(!)}(g_{(!)}\mathcal{F}))(U).$ 
\end{proof}

\section{Proper base change}
\label{sec:proper-base-change}
\begin{Bemerkungl}
 In the following we will often work with a cartesian diagram
  \begin{equation*}
    \xymatrix{\kart 
      W \ar[r]^{q}\ar[d]_{g} & Y\ar[d]^{f}\\
      Z \ar[r]^{p} & X 
    }  
  \end{equation*}
  of topological spaces and continuous maps. Let us refer to it as
``given a cartesian diagram $f\circ q=p\circ g$'' without mentioning the
spaces involved.\label{cDi} 
\end{Bemerkungl}

\begin{Lemma}
  \label{l:BaWW1}
  Given a cartesian diagram $f\circ q=p\circ g$ of topological spaces as in
  \ref{cDi} 
  the identity
  $f_{(\ast)} \circ q_{(\ast)}=
  p_{(\ast)} \circ g_{(\ast)}$ induces a morphism 
  $$
  f_{(!)} \circ q_{(\ast)}\ra p_{(\ast)} \circ g_{(!)}$$
  of functors $\Sh(W) \ra \Sh(X).$
\end{Lemma}

\begin{proof}
  Let $\mathcal{G}$ be a sheaf on $W.$ It is enough to prove the
  claim on global sections, i.\,e.\ that the identity 
  $\Gamma f_{(\ast)} q_{(\ast)} \mathcal{G} = \Gamma
  p_{(\ast)}g_{(\ast)} \mathcal{G}$ induces a morphism 
  $\Gamma f_{(!)} q_{(\ast)} \mathcal{G} \ra \Gamma
  p_{(\ast)}g_{(!)} \mathcal{G}.$ Similarly as before, we interpret
  $\Gamma f_{(!)} 
  q_{(\ast)} 
  \mathcal{G}$ as the group of all $s \in \Gamma\mathcal{G}$ with
  $f \colon 
  \overline{q(\supp s)} \ra X$ proper.
  The cartesian diagram 
  $$\xymatrix{\kart
    q^{-1} (\overline{q (\supp s)}) \ar[r]\ar[d]_-{g}
    & \overline{q(\supp s)}\ar[d]^-{f}\\
    Z \ar[r] & X }$$
  and \ref{VSU}, \ref{rem:prop-BC} show that
  $g \colon  (\supp s) \ra Z$ is proper. Hence $s
  \in \Gamma p_{(\ast)} g_{(!)} \mathcal{G}$.
\end{proof}

\begin{Bemerkungl}\label{BaWWnc}
  Given a cartesian diagram $f\circ q=p\circ g$ of topological spaces as in
  \ref{cDi} 
  there is a canonical morphism
  $$p^{(\ast)} f_{(!)}  \ra g_{(!)} q^{(\ast)} $$
  of functors
  obtained as follows: the adjunction
  $(q^{(\ast)}, q_{(\ast)})$
  and \ref{l:BaWW1} provide morphisms 
  $f_{(!)} \ra f_{(!)}  q_{(\ast)}q^{(\ast)} \ra
  p_{(\ast)}g_{(!)}q^{(\ast)}$ so that we can use the adjunction
  $(p^{(\ast)}, p_{(\ast)}).$
\end{Bemerkungl}

\begin{theorem}[{\bf Proper base change}]
  \label{t:BaWe} 
  Let a cartesian diagram $f\circ q=p\circ g$ of topological spaces as in
  \ref{cDi}  be given.
  Assume that the vertical maps $f$ and $g$ are separated
  locally proper. Then the morphism constructed in 
  \ref{BaWWnc} is an isomorphism
  \begin{equation*}
    p^{(\ast)} \circ f_{(!)}  \sira
    g_{(!)} \circ q^{(\ast)}    
  \end{equation*}
  of functors $\Sh(Y) \ra \Sh(Z)$.
\end{theorem}

\begin{Bemerkungl}
  If $f$ and $g$ are even separated and proper, 
  we obtain by \ref{p:sele} an
  isomorphism
  $p^{(\ast)} f_{(\ast)}  \sira g_{(\ast)} q^{(\ast)}$ of
  functors.
\end{Bemerkungl}

\begin{proof}
  We need to show that 
  our morphism of functors 
  evaluated at $\mathcal{F} \in \Sh(Y)$ 
  is an isomorphism
  $$p^{(\ast)} f_{(!)} \mathcal{F} \sira g_{(!)}
  q^{(\ast)} \mathcal{F}.$$ 
  A morphism of sheaves is an isomorphism if and only if it
  induces isomorphisms on all stalks. Hence we can assume without
  loss of generality that
  $Z$ consists of a single point $x \in X.$
  In this case the theorem claims that the obvious map
  is an isomorphism
  $$(f_{(!)}\mathcal{F})_{x} \sira \Gamma_{!}
  (f^{-1}(x); \mathcal{F}).$$

  Injectivity: 
  Let $U \co X$ be an open neighborhood of $x$ and $s
  \in \mathcal{F}(f^{-1}(U))$ a section with $(\supp s) \ra
  U$ proper but $s|_{f^{-1}(x)} =0$. Then $x$ is not in the
  closed subset $f(\supp s) \subset
  U$. If $V\co U$ is the open complement of this subset,
  the restriction of our section $s \in (f_{(!)}
  \mathcal{F})(U)$
  to $V$ alias $f^{-1}(V)$ is zero. 

  Surjectivity: 
  Let $s \in \Gamma_{!} (f^{-1}(x); \mathcal{F}).$
  Denote by $\bar{\mathcal{F}}$ the \'etale space over $Y$
  associated to $\mathcal{F}.$ 
  We view $s$ as a continuous section
  $s\colon f^{-1}(x) \ra \bar{\mathcal{F}}$
  with compact support $K:=\supp s \subset f^{-1}(x).$ 
  Since $f$ is separated, \ref{SRH} and
  \ref{p:FSK} allow us to 
  extend $s|_K$ to a  
  continuous section $\hat{s}\colon C \ra \bar{\mathcal{F}}$
  where $C \co Y$ is an open neighborhood of $K.$
  By shrinking $C$ we can assume in addition that 
  $\hat{s}$ and $s$ coincide on $C \cap f^{-1}(x)$
  (the subset of $C \cap f^{-1}(x)$ where the restrictions of $s$
  and $\hat{s}$ coincide is open in $C \cap
  f^{-1}(x)$ and contains $K,$ hence is of the form $C' \cap C
  \cap f^{-1}(x)$ for some open subset $K \subset C' \co Y$;
  replace $C$ by 
  $C'$).
  Hence $s$ and $\hat{s}$ glue to a continuous section 
  $C \cup f^{-1}(x) \ra \bar{\mathcal{F}}.$

  We claim that there is a subset $B \subset Y$ and an open
  subset $U \co X$
  with
  $K \subset B^{\circ} \subset B \subset C$
  and $f(B) \subset U$ such that $f\colon B \ra U$ is proper.

  Indeed, since $f$ is locally proper, any $y \in K$ has a
  neighborhood $A_y$ in $Y$ contained 
  in $C$ such that there is an open neighborhood  
  $U_{y} \co X$ of $f(y)=x$ with $f(A_y) \subset U_y$ and 
  $f\colon A_{y} \ra U_{y}$ proper.
  Since $f$ is separated, $A_y \subset f^{-1}(U_y)$ is a closed
  subset by 
  \ref{VSU},
  \ref{sep-BC},
  \ref{l:SaA}. 
  Finitely many $A_{y}^\circ$ cover $K,$ and we
  define $U$ as the intersection of the corresponding sets
  $U_{y}$ and $B$ as the union of the corresponding sets
  $A_{y}$ intersected with $f^{-1}(U)$.
  Then $f \colon B \ra U$ is proper by \ref{VUAa}, proving the
  claim.
  
  Certainly $B^{\circ}\co f^{-1}(U)$ is an open subset, and $B
  \subset 
  f^{-1} (U)$ is closed by \ref{l:SaA}. 
  Let
  $\partial B = B \setminus B^{\circ}$ 
  be the boundary of $B$ with respect to $f^{-1}(U).$
  Since $(\supp \hat{s})$ is closed in $C,$ the intersection 
  $(\supp \hat{s}) \cap \partial B$ is closed in $B$ and hence 
  the map $f\colon (\supp \hat{s}) \cap \partial B \ra U$ is
  proper. Its image does not contain $x$ because
  $f^{-1}(x) \cap (\supp \hat{s}) \cap \partial B = K
  \cap \partial B = \emptyset.$
  By replacing $U$ with its open subset $U \setminus
  f((\supp \hat{s}) \cap \partial B)$ we can assume
  that 
  $(\supp \hat{s}) \cap \partial B =\emptyset.$
  Hence the continuous map $\hat{s}|_B\colon  B \ra \bar{\mathcal{F}}$
  and the zero section $f^{-1}(U)\setminus B^\circ \ra
  \bar{\mathcal{F}}$ coincide on $\partial B$ and glue to a
  continuous map $f^{-1}(U) \ra \bar{\mathcal{F}}$
  alias an element of $\mathcal{F}(f^{-1}(U))$ whose support is
  equal to $(\supp \hat{s}) \cap B$ and hence proper over $U.$
  We deduce surjectivity.
\end{proof}


\begin{Lemma}
  \label{l:VTVT}
  If $f\colon  Y \ra X$ is a separated  locally proper map,
  the proper direct image functor $f_{(!)}\colon \Sh(Y) \ra
  \Sh(X)$ commutes with 
  filtered co\-limits.
\end{Lemma}


\begin{proof}
  Let $\mathcal{F} \colon  I \ra \Sh(Y)$ be a filtered diagram.
  We claim that the obvious morphism is an isomorphism
  \begin{displaymath}
    \varinjlim \left( f_{(!)} \mathcal{F}_i\right)
    \sira
    f_{(!)} \left(\varinjlim
      \mathcal{F}_i\right).
  \end{displaymath}
  Since the inverse image functor of sheaves commutes with
  arbitrary colimits, proper base change
  \ref{t:BaWe} allows to reduce to the case that $X$ consists of a
  single point. Then $Y$ is a locally compact Hausdorff space and
  the claim is well known and recalled in \ref{l:VTDLa}.
\end{proof}

\section{Derived proper direct image}
\label{sec:derived-proper-image}

\begin{Definition}
  \label{d:rel-c-soft}
  Given a continuous map $f\colon Y\ra X,$ a sheaf on $Y$
  is called
  \textbf{$f$-c-soft} 
  if its restriction to every fiber of $f$ is c-soft
  in the sense of \ref{d:cpt-coft}. 
\end{Definition}

\begin{Bemerkungl}
  This definition seems 
  to be useful just for separated locally
  proper maps.
\end{Bemerkungl}

\begin{Bemerkungl}
  \label{ikww}
  If our map $f\colon Y \ra X$ is separated, every flabby sheaf
  and in particular every injective sheaf on $Y$ is $f$-c-soft by
  \ref{p:FSK}. 
\end{Bemerkungl}

\begin{Bemerkungl}
  \label{dSK}
  Given a separated locally proper map
  $f\colon Y \ra X,$
  arbitrary direct sums alias coproducts,
  and even filtered colimits
  of $f$-c-soft sheaves
  are $f$-c-soft again.
  In fact, the inverse image functor of sheaves commutes with
  arbitrary 
  colimits, and every filtered colimit of c-soft sheaves on a
  locally compact Hausdorff space is c-soft by
  \ref{l:LKWG}.
\end{Bemerkungl}

\begin{Lemma}
  \label{l:teL}
  Let 
  $f\colon Y \ra X$
  be a separated locally proper map.
  Then any $f$-c-soft sheaf on $Y$ is
  $f_{(!)}$-acyclic.
\end{Lemma}

\begin{proof}
  Let $\mathcal{F} \in \Sh(Y)$ be $f$-c-soft and $\mathcal{F} \hra
  \mathcal{I}^\ast$ an injective resolution.
  We need to show that
  $f_{(!)}\mathcal{F}\hra f_{(!)}\mathcal{I}^\ast$ is exact.
  We test this on the stalks at an arbitrary $x \in X.$
  Let $j$ be the inclusion $f^{-1}(x) \hra Y.$ By proper base
  change~\ref{t:BaWe} it is enough to show exactness of
  $\Gamma_!(j^{(\ast)}\mathcal{F})\hra
  \Gamma_!(j^{(\ast)}\mathcal{I}^\ast).$ But this follows from 
  \ref{ikww} and
  \cite[Prop.~2.5.8, Cor.~2.5.9]{KS}.
\end{proof}

\begin{Bemerkungl}\label{BKWe}
  Let $g \colon Z \rightarrow Y$ be a continuous map of locally
  compact Hausdorff spaces.
  If $\mathcal{F} \in \Sh(Z)$ is c-soft so is
  $g_{(!)} \mathcal{F} \in \Sh(Y)$ by \cite[Prop. 2.5.7]{KS}.
\end{Bemerkungl}

\begin{Lemma}
  \label{l:EDKW}
  Let $g \colon Z \rightarrow Y$ and $f \colon Y \rightarrow X$
  be separated locally proper maps.  If $\mathcal{F} \in \Sh(Z)$
  is an $(f \circ g)$-c-soft sheaf on $Z$, its proper direct image
  $g_{(!)}\mathcal{F}$ is an $f$-c-soft sheaf on $Y$.
\end{Lemma}

\begin{proof}
  Given a point $x \in X$ consider the following diagram with
  two cartesian squares.
  \begin{displaymath}
    \xymatrix{
      g^{-1} (f^{-1} (x))\ar[d]_-{v} \ar@{^{(}->}[r]^-{k} & Z\ar[d]^-g\\
      f^{-1}(x)\ar[d]_-{u} \ar@{^{(}->}[r]^-{j} &Y\ar[d]^f \\
      {\{x\}} \ar@{^{(}->}[r]^-{i} & X
    }
  \end{displaymath}
  Proper base change~\ref{t:BaWe} shows $j^{(\ast)} g_{(!)}
  \mathcal{F} \cong v_{(!)} k^{(\ast)} \mathcal{F}$. 
  By assumption, the sheaf $k^{(\ast)} \mathcal{F}$ is c-soft
  and both $g^{-1} (f^{-1} (x))$ and $f^{-1} (x)$ are locally compact
  Hausdorff spaces. Hence \ref{BKWe} shows that $v_{(!)}
  k^{(\ast)} \mathcal{F} \cong j^{(\ast)} g_{(!)} 
  \mathcal{F}$ is c-soft. This proves the claim.
\end{proof}

\begin{Definition}
  If $f \colon Y \rightarrow X$ is a continuous map we denote the
  right derived functor $ {\op{R}} f_{(!)} \colon
  D^+(Y) \rightarrow D^+(X) $ of the left exact proper direct
  image functor $f_{(!)} \colon \Sh(Y) \rightarrow \Sh(X)$ by
  $
    f_{!} \pdef {\op{R}}f_{(!)}.
  $
\end{Definition}

\begin{theorem}[{\bf Iterated derived proper images}]
  \label{t:iterated-derived-proper-image}
  Let 
  $g \colon  Z \rightarrow Y$ and $f \colon  Y \rightarrow X$
  be separated locally proper maps.
  Then the identity 
  $f_{(!)} \circ g_{(!)} = (f \circ g)_{(!)}$ from \ref{t:DBE}
  gives rise to an isomorphism 
  \begin{equation*}
    (f \circ g)_! \sira f_! \circ g_!
  \end{equation*}
  of triangulated functors $D^+(Z) \ra D^+(X)$.
\end{theorem}

\begin{proof}
  It is enough to show that the functor $g_{(!)}$ maps injective
  sheaves to $f_{(!)}$-acyclic sheaves.
  Every injective sheaf $\mathcal{I} \in \Sh(Z)$ is
  $(f \circ g)$-c-soft by \ref{ikww}.
  Then $g_{(!)} \mathcal{I} \in \Sh(Y)$ is $f$-c-soft
  by \ref{l:EDKW} and hence $f_{(!)}$-acyclic by \ref{l:teL}.
\end{proof}

\begin{theorem}[{\bf Derived proper base change}]
  \label{t:DeBaW}
  Let a cartesian diagram $f\circ q=p\circ g$ of topological spaces as in
  \ref{cDi} be given.
  Assume that the vertical maps $f$ and $g$ are separated
  and locally proper.
  Then, in the space of triangulated functors
  $D^+(Y) \rightarrow D^+(Z)$, 
  the obvious morphisms are isomorphisms
  $\op{R}(p^{(\ast)} \circ f_{(!)})
  \sira p^\ast \circ f_!$
  and $\op{R}(g_{(!)} \circ q^{(\ast)}) \sira g_!
  \circ q^\ast$  so that 
  proper base change~\ref{t:BaWe} yields an isomorphism
  \begin{equation*}
    p^\ast \circ f_! \sira g_! \circ q^\ast.
  \end{equation*}
\end{theorem}

\begin{proof}
  If $\mathcal{I}$ is an injective sheaf on $Y,$
  it is $f$-c-soft by \ref{ikww},
  its inverse
  image $q^{(\ast)}\mathcal{I}$ on $W$ is $g$-c-soft and hence 
  $g_{(!)}$-acyclic by \ref{l:teL}.
  This  shows that
  $\op{R}(g_{(!)} \circ
  q^{(\ast)}) \ra g_! \circ q^\ast$ is an isomorphism.
  The rest of the proof is obvious.
\end{proof}
\section{Projection formula}\label{PF} 

\begin{Bemerkungl}
  \label{NatT}
  Let
  $f \colon Y \rightarrow X$ 
  be a continuous map and $\mathcal{F} \in \Sh(Y)$
  and $\mathcal{G} \in \Sh(X).$
  Then we obtain a canonical morphism
  \begin{equation*}
    (f_{(\ast)} \mathcal{F}) \otimes \mathcal{G} \rightarrow 
    f_{(\ast)} (\mathcal{F} \otimes f^{(\ast)} \mathcal{G})
  \end{equation*}
  of sheaves  by adjunction from the composition
  $
  f^{(\ast)} ( f_{(\ast)}\mathcal{F} \otimes \mathcal{G}) \sira
  f^{(\ast)}f_{(\ast)} \mathcal{F} \otimes f^{(\ast)} \mathcal{G}\ra 
  \mathcal{F} \otimes f^{(\ast)} \mathcal{G}
  $, where the first morphism is  the usual
  isomorphism encoding the compatibility of tensor product and
  inverse image and the second morphism comes from the
  adjunction.
  This morphism obviously induces a morphism
  \begin{equation*}
    (f_{(!)} \mathcal{F}) \otimes \mathcal{G} \rightarrow 
    f_{(!)} (\mathcal{F} \otimes f^{(\ast)} \mathcal{G})
  \end{equation*}
  on the level of proper direct images.
\end{Bemerkungl}

\begin{theorem}[{\bf Projection formula}]
  \label{t:ProFor} 
  Let $f\colon  Y \ra X$ be a separated  locally proper map
  and  $\mathcal{F}\in\Sh(Y)$ and
  $\mathcal{G}\in\Sh(X).$ 
  If $\mathcal{G}$ is flat the morphism from \ref{NatT} is an
  isomorphism
  \begin{equation*}
    (f_{(!)} \mathcal{F}) \otimes \mathcal{G} \sira
    f_{(!)} (\mathcal{F} \otimes f^{(\ast)} \mathcal{G}).
  \end{equation*}
 If $\mathcal{F}$ is in addition assumed to 
  be $f$-c-soft, the same is true for
  $\mathcal{F} \otimes  f^{(\ast)}\mathcal{G}$. 
\end{theorem}

\begin{proof}
  Proper base change~\ref{t:BaWe} reduces the first claim to the
  case 
  that $X$ consists of a single point which is known, see
  \cite[Prop.~2.5.13]{KS}. The last claim follows from
  \cite[Prop.~2.5.12]{KS}. 
\end{proof}

\begin{theorem}[{\bf Derived projection formula}]
  \label{t:ANNO} 
  Let $f \colon  Y \rightarrow X$ be separated  locally
  proper.
  Then there is a canonical isomorphism
  \begin{equation*}
    (f_{!} \mathcal{F}) \otimes^{\op{L}} \mathcal{G}
    \sira 
    f_{!} (\mathcal{F} \otimes^{\op{L}} f^\ast \mathcal{G})
  \end{equation*}
  which is natural in $\mathcal{F} \in D^+(Y)$
  and $\mathcal{G} \in D^+(X).$
\end{theorem}

\begin{proof}
  We can assume that $\mathcal{F}$ is a bounded below complex of
  injective sheaves and that $\mathcal{G}$ is a bounded below
  complex of flat sheaves. Then the claim follows from
  the underived projection formula~\ref{t:ProFor} and
  \ref{l:teL}. 
\end{proof}

\section{Verdier duality}
\label{sec:verdier-duality}


\begin{Bemerkungl}
  If $Y$ is a topological space,
  $\mathcal{G}$ a sheaf on $Y$ 
  and $j\colon V \co Y$ the
  embedding of an open subset we define
  $\mathcal{G}_{V\subset Y}
  \pdef j_{(!)} j^{(\ast)} \mathcal{G} = j_{(!)}
  j^{(!)}\mathcal{G}\in\Sh(Y),$ cf.\ \ref{rem:shriek-embedding}.
  Since both $j_{(!)}$ and $j^{(\ast)}=j^{(!)}$ are exact we
  have
  $\mathcal{G}_{V\subset Y}
  \pdef j_{!} j^{\ast} \mathcal{G}$ in the derived category.
  We write $\DZ_{V \subset Y}$ instead of $(\DZ_Y)_{V \subset
    Y}=j_{(!)}\DZ_V$. 
\end{Bemerkungl}

\begin{Lemma}
  \label{l:Kfaz}
  Let $f \colon Y \rightarrow X$ be separated and locally proper.
  Then a sheaf $\mathcal{G} \in \Sh(Y)$ is 
  $f$-c-soft if and only if the sheaf
  $\mathcal{G}_{V\subset Y}$ is $f_{(!)}$-acyclic for all open
  subsets $V \co Y.$ 
  In particular, if $\mathcal{G}$ is $f$-c-soft then so is
  $\mathcal{G}_{V\subset Y}.$
\end{Lemma}

\begin{proof}
  Denote by $i=i_x\colon  \{x\} \hookrightarrow X$ the
  embedding of 
  a point $x\in X$. A sheaf $\mathcal{F}\in \Sh(Y)$ is 
  $f_{(!)}$-acyclic if and only if 
  $\mathcal{H}^\nu i_x^\ast f_! \mathcal{F}\neq 0$ implies $\nu=0$
  for all $\nu \in \DZ$ and $x\in X$.

  Let $j\colon V \co Y$ be the inclusion of an open subset and
  consider  
  for arbitrary $x \in X$ the
  following 
  diagram with two cartesian squares.
  \begin{displaymath}
    \xymatrix{
      V\cap f^{-1} (x)\ar[d]_{u} \ar[rr]^l && V \ar[d]^j\\
      f^{-1} (x)\ar[d]_{g} \ar@{^{(}->}^k[rr] && Y\ar[d]^f\\
      \{x\} \ar@{^{(}->}^i[rr] &&X
    }
  \end{displaymath}
  Several proper base changes \ref{t:DeBaW} show
  \begin{equation*}
    i^\ast f_!\mathcal{G}_{V\subset Y}\cong g_! k^\ast
    \mathcal{G}_{V\subset Y} = g_! k^\ast j_! j^\ast \mathcal{G}
    \cong g_! 
    u_!l^\ast j^\ast \mathcal{G} \cong g_!u_!u^\ast k^\ast
    \mathcal{G}. 
  \end{equation*}
  Using the above criterion, $\mathcal{G}_{V\subset Y}$ is
  $f_{(!)}$-acyclic if and only if the sheaf 
  $(k^\ast \mathcal{G})_{(V\cap f^{-1} (x))\subset f^{-1} (x)}$ 
  is $\Gamma_!$-acyclic for every $x \in X.$
  Now use
  \cite[Exercise~II.6]{KS}.
  For the last claim use that 
  proper base change~\ref{t:BaWe} 
  implies that
  $(\mathcal{G}_{U \subset X})_{V \subset
    X} \cong \mathcal{G}_{U \cap V \subset X}$ for open subsets
  $U$ and $V$ of $X.$ 
\end{proof}

\begin{Lemma}
  \label{l:f-c-soft-criterion}
  Let $f\colon Y \ra X$ be a separated locally proper map
  with 
  $f_{(!)}$ of finite cohomological
  dimension $ \leq d.$
  If 
  \begin{equation*}
    \mathcal{G}^0 \ra \cdots \ra \mathcal{G}^{d-1} \ra
    \mathcal{G}^d \ra 0
  \end{equation*}
  is an exact sequence in $\Sh(Y)$  and $\mathcal{G}^0,
  \mathcal{G}^1, \dots, \mathcal{G}^{d-1}$ are 
  $f$-c-soft then so is $\mathcal{G}^d$.
\end{Lemma}

\begin{proof}
  We use \ref{l:Kfaz}.
  Let $V \co Y$ be an open subset.
  We have to show that $(\mathcal{G}^d)_{V\subset Y}$ is
  $f_{(!)}$-acyclic. 
  The sequence 
  \begin{equation*}
    (\mathcal{G}^0)_{V \subset Y} \ra \cdots \ra
    (\mathcal{G}^{d-1})_{V \subset Y} \ra
    (\mathcal{G}^d)_{V \subset Y} \ra 0  
  \end{equation*}
  is exact in $\Sh(Y)$ and 
  all sheaves $(\mathcal{G}^0)_{V \subset Y}, \dots,
  (\mathcal{G}^{d-1})_{V \subset Y}$ are $f_{(!)}$-acyclic.
  Let $\mathcal{K}^i$ be the kernel of $\mathcal{G}^i \ra
  \mathcal{G}^{i+1}$ for $i=0, \dots, d-1.$
  For $p>0$ we obtain isomorphisms
  \begin{equation*}
    \op{R}^p\! f_{(!)}((\mathcal{G}^d)_{V \subset Y}) 
    \sira 
    \op{R}^{p+1}\! f_{(!)}(\mathcal{K}^{d-1}_{V \subset Y}) 
    \sira \dots \sira
    \op{R}^{p+d}\! f_{(!)}(\mathcal{K}^{0}_{V \subset Y})=0.
  \end{equation*}
  This shows that $(\mathcal{G}^d)_{V\subset Y}$ is
  $f_{(!)}$-acyclic. 
\end{proof}

\begin{Lemma}
  \label{l:tensorflatfsoftfsoft}
  Let $f\colon Y \ra X$ be a separated locally proper map
  with 
  $f_{(!)}$ of finite cohomological
  dimension.
  Let $\mathcal{K} \in \Sh(Y)$ be a flat and $f$-c-soft sheaf.
  Then $\mathcal{G} \otimes \mathcal{K}$ is $f$-c-soft for  
  any sheaf $\mathcal{G} \in \Sh(Y)$.
\end{Lemma}

\begin{proof}
  Any sheaf $\mathcal{G}$ on $Y$ has a resolution
  \begin{equation*}
    \dots \ra \mathcal{G}^{-1} \ra \mathcal{G}^0 \ra
    \mathcal{G} 
    \ra 0
  \end{equation*}
  where each $\mathcal{G}^j$ is a direct sum of sheaves of the
  form 
  $\DZ_{V \subset Y}$ with $V \co Y$ open.
  We obtain an exact sequence
  \begin{equation*}
    \dots \ra
    \mathcal{G}^{-1}\otimes \mathcal{K} \ra
    \mathcal{G}^0\otimes \mathcal{K} \ra
    \mathcal{G}\otimes \mathcal{K} 
    \ra 0.
  \end{equation*}
  By 
  \ref{l:f-c-soft-criterion}
  it is sufficient to show that each
  $\mathcal{G}^j \otimes \mathcal{K}$ is $f$-c-soft.
  But 
  $\mathcal{G}^j \otimes \mathcal{K}$ is isomorphic to a direct
  sum of sheaves 
  of the form 
  $\DZ_{V \subset Y}
  \otimes \mathcal{K} \cong \mathcal{K}_{V \subset Y}$ which are
  $f$-c-soft 
  by \ref{l:Kfaz}, so that we can use \ref{dSK}.
\end{proof}
\begin{Proposition}
  \label{p:fusk}
 Let $f \colon Y \ra X$ be a separated locally proper map
  with 
  $f_{(!)}$ of finite cohomological
  dimension, and 
  let $\mathcal{K} \in \Sh(Y)$ be a flat and $f$-c-soft sheaf.
  Then the functor
  \begin{equation*}
    f_{(!)}^{\mathcal{K}} := 
    f_{(!)}(- \otimes
    \mathcal{K}) \colon
    \Sh(Y) \ra
    \Sh(X)
  \end{equation*}
 preserves colimits
  and therefore by 
\ref{c:rightadjointex} 
admits  a right adjoint $f^{(!)}_{\mathcal{K}}$.
Furthermore $f_{(!)}^{\mathcal{K}}$ is exact and therefore its right
adjoint   $f^{(!)}_{\mathcal{K}}$ maps injective sheaves to injective sheaves.
\end{Proposition}

\begin{Bemerkungl}
  Any morphism of functors induces a morphism in the opposite direction
between the adjoint functors, if they exist. In particular, any 
morphism $\mathcal K\ra\mathcal L$ of flat $f$-c-soft sheaves 
will lead to a morphism $f^{(!)}_{\mathcal L}\ra f^{(!)}_{\mathcal K}$. 
\end{Bemerkungl}

\begin{proof}
  The functor $(- \otimes \mathcal{K})$ preserves colimits, is
  exact, and maps every sheaf to an $f$-c-soft sheaf
  by \ref{l:tensorflatfsoftfsoft}.
  Then $f^\mathcal{K}_{(!)}$ is exact
  by \ref{l:teL}
  and preserves colimits, because it preserves
  filtered colimits by
  \ref{l:VTVT}, in particular direct sums, and is right
  exact.  
  Therefore, we can apply \ref{c:rightadjointex}.
The remaining claim is obvious.
\end{proof}


    
    


    

\begin{theorem}[{\textbf{Verdier duality}}]
  \label{t:VdD-bounded}
  Let $f\colon Y \ra X$ be a separated locally proper map 
  with $f_{(!)}\colon \Sh(Y) \ra \Sh(X)$ of finite cohomological
  dimension.  Then the derived proper direct image functor
  $f_{!}\colon D^+(Y) \rightarrow D^+(X)$ has a right adjoint functor
  \begin{equation*}
    f^! \colon  D^+(X) \rightarrow D^+(Y). 
  \end{equation*}
\end{theorem}

\begin{proof}
 Let $d$ be the cohomological dimension of $f_{(!)}$.
 Let $\DZ_Y\hra \mathcal{K}$ alias
  \begin{equation*}
    0 \ra \DZ_Y \hra \mathcal{K}^0 \ra \mathcal{K}^1 \ra \dots
    \ra \mathcal{K}^d \ra 0  
  \end{equation*}
  be the Godement resolution truncated in degree $d$:
  by this we mean that each $\mathcal{K}^i$ for $i =0, \dots,
  d-1$ is the sheaf of not necessarily continuous sections of
  the cokernel of the previous map and that $\mathcal{K}^d$ is
  the cokernel of the previous map.
  Then $\mathcal{K}$ consists of flat sheaves, by
  \ref{l:ZFe}, and even of
  $f$-c-soft sheaves,
  by \ref{ikww}
  and \ref{l:f-c-soft-criterion}.

  Given $\mathcal{F} \in C^+(X)$ we construct a double
  complex 
  with entries $f^{(!)}_{\mathcal{K}^{-p}}(\mathcal{F}^q)$
by applying our functors from \ref{p:fusk}.  We 
  denote its total complex by $f^{(!)}_\mathcal{K}(\mathcal{F}).$ 
  In this way we obtain a functor
  \begin{equation*}
    f^{(!)}_{\mathcal{K}}\colon  C^+(X)\ra C^+(Y)  
  \end{equation*}
  which is right
  adjoint to the functor
  \begin{equation*}
    f_{(!)}^\mathcal{K}:= f_{(!)}(-
    \otimes \mathcal{K}) \colon C^+(Y) \ra C^+(X)
  \end{equation*}
  and transforms complexes of injective sheaves to complexes of
  injective sheaves by \ref{p:fusk}.

  Let $\mathcal{G} \in C^+(Y)$ be arbitrary. 
  Since $\mathcal{K}$ is a bounded complex of flat sheaves we can
  assume that
  $\mathcal{G} \otimes^{\op{L}} \mathcal{K} =
  \mathcal{G} \otimes \mathcal{K}$ in $D^+(Y)$.
  Since $\mathcal{K}$ is a bounded complex of flat and $f$-c-soft
  sheaves, 
  \ref{l:tensorflatfsoftfsoft} 
  and \ref{l:teL} show that
  $\mathcal{G} \otimes \mathcal{K}$ consists of 
  $f$-c-soft and hence 
  $f_{(!)}$-acyclic sheaves.
  The quasi-isomorphism $\DZ_Y \ra \mathcal{K}$ 
  then shows that the obvious morphisms
  \begin{equation*}
    f_{(!)}^\mathcal{K}(\mathcal{G}) = f_{(!)} (\mathcal{G}
    \otimes \mathcal{K}) 
    \ra 
    f_{!} (\mathcal{G} \otimes \mathcal{K})
    =f_{!} (\mathcal{G} \otimes^{\op{L}} \mathcal{K})
    \leftarrow f_{!} (\mathcal{G})
  \end{equation*}
  are isomorphisms in $D^+(X)$. 
  Let
  \begin{equation*}
    f^!_{\mathcal{K}}\colon D^+(X) \ra D^+(Y)  
  \end{equation*}
  be the right derived functor of $f^{(!)}_\mathcal{K}\colon K^+(X) \ra
  D^+(Y)$. As usual it may be computed using injective
  resolutions. 

  Now let $\mathcal{G}\in C^+(Y)$ be arbitrary and
  $\mathcal{F}\in
  C^+(X)$ a complex of injective sheaves.
  Then the facts stated above show that all maps in the following
  diagram are isomorphisms.
  \begin{equation*}
    \xymatrix@R3ex{
      {\Hom_{K(X)}(f_{(!)} (\mathcal{G} \otimes \mathcal{K}),
        \mathcal{F})} \ar[d]^-{\sim}
      &
      {\Hom_{K(Y)}(\mathcal{G}, f^{(!)}_{\mathcal{K}}(\mathcal{F}))}
      \ar[l]_-{\sim}
      \ar[ddd]^-{\sim}\\
      {\Hom_{D(X)} (f_{(!)} (\mathcal{G} \otimes \mathcal{K}),
        \mathcal{F})} \\
      {\Hom_{D(X)} (f_{!} (\mathcal{G} \otimes^{\op{L}} \mathcal{K}),
        \mathcal{F})} \ar[u]_-{\sim} \ar[d]^-{\sim} \\
      {\Hom_{D(X)}(f_{!}(\mathcal{G}),\mathcal{F})} 
      &
      {\Hom_{D(Y)} (\mathcal{G}, f^!_{\mathcal{K}}(\mathcal{F}))}
    }
  \end{equation*}
  All maps are natural in complexes $\mathcal{G} \in C^+(Y)$ and
  complexes of injective 
  sheaves
  $\mathcal{F} \in C^+(X)$.
  This yields the desired adjunction
  by setting $f^!=f^!_\mathcal{K}$.
\end{proof}

 \section{The case of unbounded derived categories}
\begin{Bemerkungl}
  We finally explain how our results
  generalize 
  to unbounded derived categories as soon as the relevant
  involved maps $f$
  are separated locally proper with $f_{(!)}$ of finite
  cohomological dimension.
\end{Bemerkungl}


\begin{Bemerkungl}
  The derived functors $f^*$, $f_*$, $\op{R}\!\!\sheafHom$,
  $\otimes^{\op{L}}$ all exist on the level of unbounded derived
  categories (see \cite[Ch.~18]{KS-cat-sh}) and
  so does $f_!$ by \ref{rem:ShX-Grothendieck}
  and \ref{l:compute-right-derived}. \end{Bemerkungl}

\begin{theorem}[{cf.\ \cite[Thm.~B]{spaltenstein}}]
  \label{t:comp-bc-proj}
  Let $f\colon Y \ra X$ be a separated locally proper map
  such that
  $f_{(!)}$
  has finite cohomological
  dimension. Then:
  \begin{enumerate}
  \item
    \label{enum:f-upper-shriek} 
    The derived proper direct image functor
    $f_{!}\colon D(Y) \rightarrow D(X)$ has a right adjoint functor
    $f^! \colon  D(X) \rightarrow D(Y)$. 
  \item
    \label{enum:comp}
    If $g\colon Z \ra Y$ is another separated locally proper
    map such that
    $g_{(!)}$ has finite cohomological
    dimension 
    then $f \circ g$ is separated locally proper with
    $(f \circ g)_{(!)}$ of finite cohomological dimension and
    there are 
    isomorphisms 
    $(f \circ g)_! \sira f_! \circ g_!$ and
    $(f \circ g)^! \sila g^! \circ f^!$ of triangulated
    functors. 
  \item 
    \label{enum:properbasechange}
    Let a cartesian diagram $f\circ q=p\circ g$  of topological
    spaces as in 
    \ref{cDi} be given.
    Then $g$ is separated locally proper with $g_{(!)}$ of finite
    cohomological 
    dimension and there are isomorphisms
    \begin{align*}
      p^\ast \circ f_!  \sira g_! \circ q^\ast \quad
      \text{resp.} \quad
      f^! \circ p_*  \sila q_* \circ g^!
    \end{align*}
    in the space of triangulated functors
    from $D(Y)$ to $D(Z)$ resp.\ from $D(Z)$ to $D(Y)$.
  \item
    \label{enum:one-map} 
    For all $\mathcal{A} \in D(X)$ and $\mathcal{B}$,
    $\mathcal{C} \in D(Y)$ there are natural isomorphisms
    \begin{align*}
      (f_{!} \mathcal{A}) \otimes^{\op{L}} \mathcal{B}
      & \sira
      f_{!} (\mathcal{A} \otimes^{\op{L}} f^\ast \mathcal{B}),\\
      \op{R}\!\!\sheafHom(f_! \mathcal{A}, \mathcal{B}) 
      & \sila f_*
      \op{R}\!\! \sheafHom(\mathcal{A}, f^! \mathcal{B}),\\ 
      f^!\op{R}\!\!\sheafHom(\mathcal{B}, \mathcal{C}) 
      & \sila \op{R}\!\! \sheafHom (f^*\mathcal{B}, f^!
      \mathcal{C}). 
    \end{align*}
\end{enumerate}
\end{theorem}
\begin{Bemerkungl}
  Verdier duality for unbounded derived
  categories has been proved in \cite{spaltenstein} 
  for continuous maps $f$ between locally compact Hausdorff spaces
  that satisfy a condition that is slightly weaker than
  our condition on the cohomological dimension of $f_{(!)}$.
  It is also 
  possible to deduce it from 
  Brown representability for well generated 
  triangulated categories
  \cite[Thm.~1.17]{neeman-tricat}, 
  \cite[Thm.~0.2]{neeman-sheaves-manifold}.
\end{Bemerkungl}

\begin{proof}
  \ref{enum:f-upper-shriek}: The proof of Verdier
  duality in the bounded case given in~\ref{t:VdD-bounded} 
generalizes easily: When we omit all upper indices $+$ there,
 \ref{l:compute-right-derived}.\ref{enum:RF-for-fin-cohodim}
  and the quasi-isomorphism $\DZ_Y \ra \mathcal{K}$ 
  still show that the obvious morphisms
  \begin{equation*}
    f_{(!)}^\mathcal{K}(\mathcal{G}) = f_{(!)} (\mathcal{G}
    \otimes \mathcal{K}) 
    \ra 
    f_{!} (\mathcal{G} \otimes \mathcal{K})
    =f_{!} (\mathcal{G} \otimes^{\op{L}} \mathcal{K})
    \leftarrow f_{!} (\mathcal{G})
  \end{equation*}
  are isomorphisms in $D(X)$. 
Thus $f_{(!)}^{\mathcal{K}}$ maps acyclic complexes to
acyclic
complexes,
and this implies that the functor \begin{equation*}
    f^{(!)}_{\mathcal{K}}\colon  C(X)\ra C(Y)  
  \end{equation*}
maps h-injective complexes to h-injective complexes.
Apart from these additions, the argument is the same.
  
  It is enough to prove the first of the isomorphisms stated in
  each of the claims 
  \ref{enum:comp}, \ref{enum:properbasechange},
  \ref{enum:one-map} because the remaining isomorphisms are formal
  consequences using the Yoneda
  lemma and various adjunctions, most prominently 
  the Verdier duality adjunction $(f_!,f^!)$.

  \ref{enum:comp}: 
  The composition $f \circ g$ is separated
  locally proper by 
  \ref{rem:comp-locprop}
  and \ref{sep-comp}; moreover, $(f \circ g)_{(!)}$ has finite
  cohomological dimension by
  \ref{t:iterated-derived-proper-image}
  and \ref{c:range-f-shriek}.
  Let $\mathcal{F} \in C(Z)$ be a fibrant
  object. Then all components of $\mathcal{F}$ are injective
  sheaves by \ref{rem:compare-Beke-KS-cat-sh}.
  As in the proof of
  \ref{t:iterated-derived-proper-image} we see that all
  components of
  $g_{(!)}\mathcal{F}$ are $f_{(!)}$-acyclic.
  Lemma~\ref{l:compute-right-derived}.\ref{enum:RF-for-fin-cohodim}
  shows that the obvious morphism
  \begin{equation*}
    f_{(!)}(g_{(!)}\mathcal{F}) \ra f_!(g_{(!)}\mathcal{F})
  \end{equation*}
  is an isomorphism in $D(X).$
  Since $\mathcal{F}$ is h-injective 
  we have 
  $g_{(!)}\mathcal{F} \sira g_{!}\mathcal{F}$ in $D(Y)$
  and
  $f_{(!)}(g_{(!)}\mathcal{F})= (f \circ
  g)_{(!)}(\mathcal{F}) \sira 
  (f \circ
  g)_{!}(\mathcal{F})$ in $D(X)$ where
  \ref{t:DBE} is used for the equality.


  \ref{enum:properbasechange}: Clearly, $g$ is separated locally
  proper by \ref{sep-BC} and \ref{l:BC-localprop}.  Derived
  proper base change \ref{t:DeBaW} reduces the question whether
  $g_{(!)}$ has finite cohomological dimension to the case that
  $Z=\{x\}$ for some $x \in X$ and $W=f^{-1}(x).$ Then $q \colon
  W \ra Y$ is an embedding and any sheaf $\mathcal{E} \in \Sh(W)$
  satisfies $\mathcal{E} \sira q^{(\ast)}q_{(\ast)}\mathcal{E}=
  q^{\ast}q_{(\ast)}\mathcal{E}.$ This and derived proper base
  change again shows $ g_! \mathcal{E} \sira g_!  q^\ast
  q_{(\ast)}\mathcal{E} \sila p^\ast f_! q_{(\ast)}\mathcal{E} $.
  This shows that $g_{(!)}$ has finite cohomological dimension.

  Let $Z$ again be arbitrary. 
  Let $\mathcal{F} \in C(Y)$ be a fibrant object.
  As in the proof
  of \ref{t:DeBaW} we see that $q^{(*)}\mathcal{F}$ 
  consists of $g_{(!)}$-acyclic sheaves. 
  Then
  \begin{equation*}
    g_{(!)}(q^{(*)}\mathcal{F}) \sira 
    g_!(q^{(*)}\mathcal{F}) \sira g_!(q^{*}\mathcal{F})
  \end{equation*}
  in $D(Z)$ where the first isomorphism comes from 
  \ref{l:compute-right-derived}.\ref{enum:RF-for-fin-cohodim}
  and the second one from the isomorphism 
  $q^{(*)}\mathcal{F} \sira q^*\mathcal{F}$ in $D(W)$.
  Now use proper base change~\ref{t:BaWe}.

  \ref{enum:one-map}: 
  Following \cite[Prop.~6.18]{spaltenstein}
  we prove the derived
  projection formula.
  From \ref{NatT}
  we obtain a natural morphism 
  \begin{equation}
    \label{eq:proj-form}
    \tag{$\star$}
    (f_{(!)} \mathcal{A}) \otimes \mathcal{B} \rightarrow 
    f_{(!)} (\mathcal{A} \otimes f^{(\ast)} \mathcal{B})
  \end{equation}
  in $C(X).$ 
  Assume that 
  $\mathcal{A}$ is a complex of $f$-c-soft sheaves on $Y$ and that
  $\mathcal{B}$ is an object of
  $\underrightarrow{\mathfrak{P}}(X)$ in the
  notation of \cite[Sect.~5]{spaltenstein}.
  In particular, 
  $\mathcal{B}$ is h-flat, and $\mathcal{A}$ consists of
  $f_{(!)}$-acyclic sheaves  
  by \ref{l:teL}.
  Hence the left hand side of \eqref{eq:proj-form}
  computes 
  $(f_{!} \mathcal{A}) \otimes^{\op{L}} \mathcal{B}$
  by
  \ref{l:compute-right-derived}.\ref{enum:RF-for-fin-cohodim}.
  On the other hand, $f^{(\ast)}\mathcal{B}$ is certainly h-flat
  so that $\mathcal{A} \otimes^{\op{L}}
  f^{(\ast)}\mathcal{B}=\mathcal{A} \otimes
  f^{(\ast)}\mathcal{B}$. 
  We claim that 
  $\mathcal{A} \otimes f^{(\ast)}\mathcal{B}$
  has $f$-c-soft components.
  If $\mathcal{B}'$ is a bounded above complex of sheaves on $X$
  whose components are direct sums of
  sheaves of 
  the form 
  $\DZ_{U \subset X}$ with $U \co X$ open, i.\,e.\ $\mathcal{B}'
  \in \mathfrak{P}(X)$ in the notation of \cite{spaltenstein},
  then all components of  
  $\mathcal{A} \otimes f^{(\ast)}\mathcal{B}'$
  are $f$-c-soft as direct sums of $f$-c-soft sheaves 
  $\mathcal{A}^p \otimes f^{(*)}\mathcal{B}'^q$, by
  \ref{dSK} and \ref{t:ProFor}. Recall that
  $\underrightarrow{\mathfrak{P}}(X)$ is the closure of 
  $\mathfrak{P}(X)$ under certain filtered colimits (see
  \cite[2.9]{spaltenstein}) in $C(X)$.
  Since $(\mathcal{A} \otimes
  -)$ commutes with filtered colimits and $f^{(*)}$ commutes with
  all colimits, this implies our claim that 
  $\mathcal{A} \otimes f^{(\ast)}\mathcal{B}$
  has $f$-c-soft components
  (using \ref{dSK} again).
  Then \ref{l:compute-right-derived}.\ref{enum:RF-for-fin-cohodim}
  again shows that the right hand side 
  of 
  \eqref{eq:proj-form}
  computes
  $f_! (\mathcal{A} \otimes f^\ast \mathcal{B})$.

  Hence it is certainly enough to show that 
  \eqref{eq:proj-form}
  is an isomorphism. Since all the functors
  $f_{(!)}$, $\otimes$ and $f^{(*)}$ commute with filtered
  colimits, we can assume without loss of generality that
  $\mathcal{B} \in \mathfrak{P}(X)$ and even that $\mathcal{B}=
  \DZ_{U \subset X}$ for some open subset $U \co X$.
  But in this case  
  \eqref{eq:proj-form} is an isomorphism by
  \ref{t:ProFor}. 
  This establishes the derived projection formula.
\end{proof}
\begin{corollary}
  \label{c:range-f-shriek}
  In the setting of \ref{t:comp-bc-proj}
  we have 
  $f_!(D^{\geq 0}(Y)) \subset D^{\geq 0}(X)$ and
  $f_!(D^{\leq 0}(Y)) \subset D^{\leq +d}(X)$
  and
  $f^!(D^{\geq 0}(X)) \subset D^{\geq -d}(Y)$.
\end{corollary}

\begin{proof}
  The first claim
  holds (for any continuous map
  $f$)
  because any object of $D^{\geq 0}(X)$ is isomorphic to a complex
  of injective sheaves whose components in degrees $<0$ are zero.
  We have seen in the above proof that $f_!(\mathcal{G}) \cong
  f_{(!)}^\mathcal{K}(\mathcal{G})$ in $D(X)$. Hence
  $f_!(D^{\leq 0}(X)) \subset D^{\leq +d}(Y)$.
  Then $f^!(D^{\geq 0}(X)) \subset D^{\geq -d}(Y)$
  follows from the adjunction $(f_!, f^!)$ or from the explicit
  construction of the functor $f^!$.

\end{proof}

\section{Reminders from topology}
\label{sec:reminders}

\label{sec:topology}


\begin{Bemerkungl}\label{DefE}
  A map $f \colon Y \ra X$ of
  topological spaces is called \textbf{proper} if it is
  continuous and  
  the map
  $f \times \op{id} \colon Y \times Z \ra X
  \times Z$ is closed for each topological space $Z.$
  All properties of proper maps we
  mention in the following can be found in
  \cite[I.\S~10]{BouTo14-french} or
  \cite[\href{http://stacks.math.columbia.edu/tag/005M}{Section~005M}]{stacks-project}.
\end{Bemerkungl}

\begin{Bemerkungl}
  A topological space $X$ is compact if and only if the constant
  map from $X$ to a space consisting of a single point is proper.
  \label{KrE}
\end{Bemerkungl}

\begin{Bemerkungl}
  \label{VSU}  
  An embedding is
  proper if and only if it is closed.   
  \label{comp-proper}
  Every composition of proper maps is proper.
\end{Bemerkungl}

\begin{Bemerkungl}
  \label{rem:prop-BC}
  The property of being proper is stable under base change:
  given a proper map $Y \ra X$ and an
  arbitrary continuous map $Z\ra X$ the map $Z \times_X Y \ra Z$
  is proper. 
  In
  particular, a map is proper if and only if it is universally
  closed. There are two other important characterizations of
  proper maps.
  A continuous map is proper if
  and only if 
  it is closed and the inverse image
  of every compact subset is compact.
  A continuous map is proper if
  and only if 
  it is closed and all its fibers are compact. 
\end{Bemerkungl}

\begin{Bemerkungl}
  Properness is local on the target:
  let $f \colon Y \ra X$ be a map and $\mathcal{U}$ an open
  covering  of $X$.  
  Then $f$ is proper if and only if the induced maps 
  $f^{-1}(U) \ra U$ are proper for all $U \in
  \mathcal{U}$.\label{prop-localtarget}
  Note however that properness is not local on the source
  (cf.\ \ref{locprop-localsource} for our usage of this
  notion). 
\end{Bemerkungl}

\begin{Bemerkungl}
  If a
  composition $f \circ g$ of continuous maps is proper and $g$ is
  surjective, then 
  $f$ is proper.\label{EWEi} 
\end{Bemerkungl}

\begin{Bemerkungl}\label{VUAa}
  If $Y\ra X$ and $Y'\ra X$ are proper maps, so is $(Y\amalg
  Y')\ra X$. In particular, given a continuous map $Z\ra X$ and
  subspaces 
  $Y,$ $Y'\subset Z$ with $Y\ra X$ and $Y'\ra X$ proper,
  the map $Y\cup Y' \ra X$ is proper by \ref{EWEi}.
\end{Bemerkungl}

\begin{Definition}
  \label{Msep}
  A map $f\colon  Y \ra X$ of topological spaces is
  called \textbf{separated} if it is continuous and the diagonal
  map $Y \ra Y \times_{X} Y$ is a closed embedding. The second
  condition is satisfied if and only if 
  the diagonal is a closed subset of $Y \times_{X} Y.$
\end{Definition}

\begin{Definition}\label{RHa}
  A subset $A$ of a topological space $Y$ is called
  \defind{relatively Hausdorff}
  if any two distinct points of $A$ have disjoint neighborhoods
  in $Y.$
\end{Definition}

\begin{Bemerkungl}
  \label{SRH}
  A continuous map is separated if and only if all its fibers are
  relatively Hausdorff.
\end{Bemerkungl}

\begin{Bemerkungl}
  The constant map from a topological space $X$ to a space
  consisting of a single point is separated if and only if $X$ is
  a Hausdorff space. 
  Every embedding is separated.
  Every composition of separated maps is separated.
  \label{sep-comp}
  \label{sep-BC}
  \label{sep-local-target}
  The property of being separated is stable under base change and
  local on the target, but not local on the source.
  If a composition $f \circ g$ of continuous maps is separated,
  $g$ is separated. In particular, any continuous map whose
  source is a
  Hausdorff space is separated.
\end{Bemerkungl}

\begin{Lemma}
  \label{l:SaA}
  If a composition $f \circ g$ of continuous maps is proper
  and $f$ is separated, then $g$ is proper.
\end{Lemma}

\begin{Bemerkungl}
  This generalizes \cite[Prop.~I.10.1.5.(c)]{BouTo14-french}.
\end{Bemerkungl}

\begin{Bemerkungl}
  If the target of $f$ consists of a single point, this lemma
  specializes to the fact that a continuous map from
  a compact space to a Hausdorff space is closed and even
  proper. 
\end{Bemerkungl}

\begin{proof}
  Let $g\colon Z\ra Y$ and $f\colon Y\ra X$. Consider the
  cartesian diagram
  $$
  \xymatrix{ \kart Z\ar[r]^-{g} \ar[d]_-{(\op{id},g)} & Y
    \ar[d]^-{(\op{id},\op{id})} \\
    Z\times_X Y \ar[r]^-{g \times \op{id}} & Y\times_X Y }
  $$
  of topological spaces.  The diagonal map $Y \hookrightarrow Y
  \times_X Y$ is a closed embedding and hence so is the map
  $(\op{id},g) \colon Z \ra Z \times_X Y$ obtained by base
  change using \ref{sep-BC}.
  The morphism $g$ is
  the composition of this latter map with the proper map $Z
  \times_X Y \ra Y$, using \ref{comp-proper} and \ref{rem:prop-BC}.
\end{proof}

\section{Reminders from sheaf theory}
\label{sec:sheaf-theory}

\begin{Proposition}
  [{\bf Extension of sections on compacta}]
  \label{p:FSK}
  Let $\mathcal{F}$ be a sheaf of sets  on a
  topological space $X.$
  If $K \subset X$ is a compact relatively Hausdorff subset of
  $X,$ any section
  $s \in \mathcal{F}(K)$ is the restriction of a
  section of $\mathcal{F}$ on an open neighborhood of $K$ in $X.$
\end{Proposition}

\begin{proof}
  Let $s \in \mathcal{F} (K).$
  For any $x \in K$ there is an open neighborhood $U_x$ of $x$ in
  $X$ and a section $s_x \in \mathcal{F}(U_x)$ such that $s_x|_{K
    \cap U_x}= s|_{K \cap U_x}.$ 
  Since $K$ is locally compact as a compact Hausdorff space,
  there is a compact neighborhood $K_x$ of $x$ in $K$ that is
  contained in $K \cap U_x.$
  Since $K$ is compact, finitely many $K_{x_1}, \dots, K_{x_n}$
  cover $K.$ Put $K_i:=K_{x_i},$ $U_i :=U_{x_i}$ and $s_i:=s_{x_i}.$
  We have $s_{i}|_{K_{i}} = s|_{K_{i}}$.

  If $n=1$ we are done. Otherwise
  compactness of $K_1 \cap K_2$ yields an open
  neighborhood $W$ of $K_1 \cap K_2$ in $U_1 \cap U_2$
  with $s_1|_W = s_2|_W$.
  Moreover,
  there are  
  disjoint open neighborhoods
  $U_i^{\prime} \co U_i$ of $K_i \setminus W$
  for $i=1,$ $2$. Here we
  use that $K$ is relatively Hausdorff and that $K_i \setminus
  W$ is compact.
  Then the three sections 
  $s_1|_{U^{\prime}_1},$ $s_2|_{U^{\prime}_2}$ and $s_1|_W = s_2|_W$
  glue to a section on
  $U^{\prime}_1 \cup U^{\prime}_2 \cup W$
  which extends $s|_{K_1\cup K_2}.$
  An easy induction finishes the proof.
\end{proof}

\begin{Lemma}
  \label{l:VTDLa}  
  On a locally compact Hausdorff space $X,$
  taking global sections with compact support commutes
  with filtered colimits,
  i.\,e.\
  \begin{displaymath}
    \varinjlim \big(\Gamma_{!} \mathcal{F}_i \big)
    \sira
    \Gamma_{!}
    \big(\varinjlim \mathcal{F}_i\big)
  \end{displaymath}
  for any filtered diagram $\mathcal{F} \colon  I \ra \Sh(X)$.
\end{Lemma}

\begin{proof}
  We use in the proof that the colimit of a system of sheaves 
  is the sheafification of the colimit in the category of
  presheaves and that the canonical morphism from a presheaf to
  its 
  sheafification induces isomorphisms on all stalks.
  We can assume without loss of generality that $I$ is a
  partially ordered set.  

  Injectivity:
  Let
  $s\in \Gamma_{!}(X; \mathcal{F}_j)$ for some $j\in
  I$. Assume that $s$ goes to zero in $\Gamma_!(X; \varinjlim
  \mathcal{F}_i).$ For every point $x\in X$ there are an open
  neighborhood 
  $U(x)$ and an index $i(x)$
  with $s\mapsto 0 \in  \mathcal{F}_{i(x)}
  (U(x))$.
  Finitely many $U(x)$ cover the compact set $(\supp s).$ Let $i$
  be bigger than the finitely many $i(x)$ involved.
  Then $s\mapsto 0 \in  \mathcal{F}_{i} (X)$ showing injectivity.

  Surjectivity for $X$ compact and Hausdorff:
  Let $s\in \Gamma(X; \varinjlim \mathcal{F}_i)$ be given.
  For every point $x\in X$ there are a neighborhood
  $U(x)$ in $X,$ an index $i(x)$ and a section $s(x)\in
  \mathcal{F}_{i(x)}(U(x))$ with $s(x)\mapsto s|_{U(x)}$. Since
  $X$ is locally compact we may assume that the neighborhoods
  $U(x)$ are compact. Since $X$ is compact there is a finite
  subset $E\subset X$ such that the $U(x)$ with $x\in E$
  cover $X.$
  Injectivity proved above shows that for any
  $x, y \in E$ there is some $j \in I$ such that the images 
  of $s(x)$ and
  $s(y)$ in $\mathcal{F}_{j}(U(x)\cap U(y))$
  coincide. Since $E$ is finite we can assume that the index $j$
  works 
  for all $x, y \in E.$
  But then the images of the $s(x)$ in $\mathcal{F}_j(U(x))$ for
  $x \in E$ glue to a global section of $\mathcal{F}_j$
  which represents the 
  inverse image of $s$ we are looking for.
  
  Surjectivity for $X$ locally compact and Hausdorff:
  Let $s\in
  \Gamma_{!}(X; \varinjlim \mathcal{F}_i)$ be given. Since
  $(\supp s)$ is compact and $X$ is locally compact and Hausdorff
  we find
  $U\subset X$ open with compact closure $\ol{U}$ and
  $(\supp s)\subset U$. 
  We have already proved surjectivity of
  $\varinjlim \big(\Gamma(\ol{U}; \mathcal{F}_i \big)
  \ra
  \Gamma\big(\ol{U};\varinjlim \mathcal{F}_i\big).$
  Hence there is an index $j$ and a section
  $\tilde s\in \Gamma(\ol{U}; \mathcal{F}_j)$
  with $\tilde s\mapsto s|\ol{U}$.
  Then $\tilde
  s|\partial\ol{U} \mapsto 0\in \Gamma(\partial \ol{U};
  \varinjlim \mathcal{F}_i)$ and injectivity proved above shows
  that 
  $\tilde s|\partial \ol{U} \mapsto 0\in \Gamma(\partial \ol{U};
  \mathcal{F}_l)$ for some $l\geq j$.
  Hence the image of $\tilde s$ in
  $\Gamma(\ol{U}; \mathcal{F}_l)$ and the zero section $0 \in
  \Gamma(X \setminus U; \mathcal{F}_l)$ glue to a global section
  $\hat s\in \Gamma_{!}( X; \mathcal{F}_l)$ with compact
  support and $\hat s\mapsto s.$ This proves the lemma.
\end{proof}

\begin{Definition}
  \label{d:cpt-coft}
  A sheaf is called \textbf{c-soft} (for compact-soft)
  if any section over a compact set comes from a global section.
\end{Definition}

\begin{Lemma}[{\bf Filtered colimits of c-soft sheaves}]
  \label{l:LKWG}  
  On a locally compact Hausdorff space $X,$ any filtered colimit
  of c-soft sheaves is c-soft.
\end{Lemma}

\begin{proof}
  Let $\mathcal{F} \colon  I \ra \Sh(X)$ be a filtered diagram of c-soft
  sheaves. 
  The inverse image functor of sheaves commutes with arbitrary
  colimits because it has a right adjoint functor.
  Let $u \colon K\hra X$ be the embedding of a compact subset.
  Hence the obvious morphism is an isomorphism
  $\varinjlim
  u^{(\ast)}\mathcal{F}_i\sira u^{(\ast)}\varinjlim
  \mathcal{F}_i$. This and \ref{l:VTDLa}
  provide isomorphisms
  $$
  \varinjlim \Gamma(K;u^{(\ast)}\mathcal{F}_i)
  \sira 
  \Gamma(K;\varinjlim u^{(\ast)}\mathcal{F}_i)
  \sira
  \Gamma(K;u^{(\ast)}\varinjlim \mathcal{F}_i)$$
  or concisely an isomorphism
  $$\varinjlim \Gamma(K;\mathcal{F}_i)
  \sira 
  \Gamma(K;\varinjlim \mathcal{F}_i).$$
  Hence 
  every section $t$ of the group on the right
  is the image of a section $s\in
  \Gamma(K;\mathcal{F}_i)$ for some $i\in I$. Since
  $\mathcal{F}_i$ is c-soft there is a global section
  $\tilde s\in \Gamma(X;\mathcal{F}_i)$
  with $\tilde{s}|_K=s.$
  The image of $\tilde s$ in
  $\Gamma(X;\varinjlim \mathcal{F}_i)$ then restricts to $t.$
\end{proof}

\begin{Lemma}
  \label{l:ZFe}
  Let $\mathcal{F}$ be a sheaf on a topological space $X$ and
  $\mathcal{G}\mathcal{F}$ the 
  sheaf of not necessarily continuous sections of $\mathcal{F}$
  from \cite[II.4.3]{godement}.  If $\mathcal{F}$ is flat, so
  are $\mathcal{G}\mathcal{F}$ and the cokernel $\op{cok}
  (\mathcal{F} \hookrightarrow \mathcal{G} \mathcal{F})$ of the
  canonical monomorphism $\mathcal{F} \hookrightarrow \mathcal{G}
  \mathcal{F}.$
\end{Lemma}

\begin{proof}
  An abelian group is flat if and only if it is torsion-free, and
  these properties are preserved under products.
  For $x \in X$ we have
  \begin{equation*}
    (\mathcal{G}\mathcal{F})_x
    =
    \lim_{\xrightarrow[x \in U]{}}(\mathcal{G}\mathcal{F})(U)
    =
    \lim_{\xrightarrow[x \in U]{}} \prod_{u \in
      U}\mathcal{F}_u 
  \end{equation*}
  where $U$ ranges over the open subsets of $X$ containing $x.$
  Assume that $\mathcal{F}$ is flat. Then all $\mathcal{F}_u$ are
  flat abelian groups and it is easy to see that 
  $(\mathcal{G}\mathcal{F})_x$ is torsion-free and hence flat.
  Hence $\mathcal{G}\mathcal{F}$ is flat.
  Since $\mathcal{F} \hra \mathcal{G}\mathcal{F}$ induces split
  injections $\mathcal{F}_x \hra (\mathcal{G}\mathcal{F})_x$ on
  each stalk the cokernel has flat stalks.
\end{proof}

\section{Representability and adjoints}

\begin{Lemma}[{\textbf{Representability Lemma}}] 
  \label{l:represent}
  Let $X$ be a topological space.   
  A functor 
  $\mu \colon  \Sh(X) \ra \Ab^\opp$
  is representable if and only if it preserves all colimits.
\end{Lemma}

\begin{Bemerkungl}
  We learnt the above lemma from \cite[Exp.~4, 1.0]{SHS}. Our proof
  completes the proof given there.
\end{Bemerkungl}

\begin{proof}
 It is clear that any representable functor commutes with
 colimits. Assume that $\mu$ is representable, i.\,e.\ there
 is a sheaf 
 $\mathcal C \in \Sh(X)$ 
 together with natural isomorphisms
 \begin{equation*}
   \mu (\mathcal F) \sira 
   \Sh_X(\mathcal F, \mathcal C)
 \end{equation*}
 for $\mathcal{F} \in \Sh(X)$. Plugging in
 $\mathcal{F}=\DZ_{U\subset X}$ for $U \co X$ open we obtain an
 isomorphism 
 $\mu (\DZ_{U \subset X}) 
 \sira \mathcal C (U)$.
 If $V \co U$ is an open subset, we obtain in this way the
 horizontal isomorphisms in the 
 commutative diagram
 \begin{equation*}
   \xymatrix{
     {\mu (\DZ_{U \subset X})} \ar[r]^-\sim \ar[d] 
     & {\mathcal{C}(U)} \ar[d]^-{\op{res}^V_U}
     \\ 
     \mu(\DZ_{V \subset X}) \ar[r]^-\sim  
     & {\mathcal{C}(V)}
   }
 \end{equation*}
 whose 
 left vertical morphism is the image under $\mu$ of the
 obvious morphism 
 $\DZ_{V \subset X} \rightarrow \DZ_{U \subset X}$. 

 Now assume that $\mu$ preserves all colimits. Define a presheaf
 $\mathcal C_\mu$ by
 \begin{equation*}
   \mathcal C_\mu (U) := \mu (\DZ_{U \subset X})
 \end{equation*}
 with restriction maps coming from the morphisms
 $\DZ_{V \subset X} \rightarrow \DZ_{U \subset X}$. 
 We first claim that $\mathcal{C}_\mu$ is a sheaf.
 Let $\mathcal{U}$ be an open covering of an open subset $V \co
 X$.
 Consider the obvious coequalizer diagram
 \begin{equation*}
   \bigoplus_{(U,U^\prime) \in \mathcal U^2} \mathbb Z_{(U \cap
     U^\prime)\subset X} 
   \begin{array}{c} \rightarrow\\[-2ex] 
     \rightarrow \end{array}
   \bigoplus_{U \in \mathcal U} \DZ_{U \subset X} 
   \rightarrow \DZ_{V \subset X}.
 \end{equation*}
 It presents $\DZ_{V \subset X}$ as a colimit. 
 Since $\mu$
 commutes with colimits we see that $\mathcal{C}_\mu$ is a
 sheaf.
 
 Now let us show that $\mathcal{C}_\mu$ represents $\mu$.
 Let $\op{Op}_X$ be the category of open subsets of $X$ with 
 inclusions as morphisms and let
 $J \colon \op{Op}_X \rightarrow  \Sh(X) $ be the functor
 $U  \mapsto  \DZ_{U \subset X}$.
 For any $\mathcal{F} \in \Sh(X)$ consider the morphisms
 \begin{equation*}
   \xymatrix{
     {\mu (\mathcal{F})} \ar[r]^-\sim
     & 
     {\op{Fun}_{(\Sh(X), \op{Ab}^{\opp})}(\mu, \Sh_X (-,
     \mathcal{F}))}  
     \ar[d]^-{\circ J}\\
     \Sh_X(\mathcal{F}, \mathcal{C}_\mu)
     &
     {\op{Fun}_{(\op{Op}_X, \op{Ab}^{\opp})} 
       (\mu \circ J, \Sh_X(J(-), \mathcal{F}))}
     \ar[l]_-\sim
   }
 \end{equation*}
 The upper horizontal morphism is the isomorphism from the Yoneda
 lemma, the 
 vertical morphism is obvious, and the lower horizontal morphism
 comes from 
 $\mu \circ J=\mathcal{C}_\mu$ and the isomorphism 
 $\Sh_X(J(-), \mathcal{F}) \sira \mathcal{F}(-)$.
 
 We obtain a morphism $\gamma \colon \mu \ra \Sh_X(-,
 \mathcal{C}_\mu)$ of functors $\Sh(X) \ra \Ab^\opp$.
 Since $\gamma$ is an isomorphism at each object of the form
 $\DZ_{U \subset X}$ and since both functors
 $\mu$ and $\Sh_X(-,
 \mathcal{C}_\mu)$ commute with colimits, $\gamma$ is an
 isomorphism and $\mathcal{C}_\mu$ represents $\mu.$
\end{proof}

\begin{corollary}[{\cite[Exp.~4, 1.1]{SHS}}]
  \label{c:rightadjointex}   
  Let $X$ be a topological space and $\mathcal{B}$ an abelian
  category. A functor $\Lambda \colon
  \Sh(X) 
  \ra \mathcal{B}$ has a right adjoint if and only if it 
  preserves colimits.
\end{corollary}

\begin{Bemerkungl}\label{edc} 
  Again, if we know a right adjoint $R$ exists, it is not difficult 
  to describe it: We just need to
  apply the adjunction isomorphism
  $\mathcal B(\Lambda \mathcal F, B)\sira \op{Sh}_X(\mathcal F, RB)$
  to the sheaf $\mathcal F=\DZ_{U\subset X}$ to find isomorphisms
  $\mathcal B(\Lambda \DZ_{U\subset X}, B)\sira (RB)(U)$ which
  describe the functor $R$ quite explicitly.
\end{Bemerkungl}

\begin{proof}
  Any left adjoint functor certainly preserves colimits.
  If $\Lambda$ preserves colimits, 
  apply 
  Lemma~\ref{l:represent} 
  to the functor
  $\mathcal{B} (\Lambda(-), B)$ 
  where $B \in \mathcal{B}$ is fixed.
\end{proof}

\section{Remarks on derived functors}

\begin{Bemerkungl}
  \label{rem:ShX-Grothendieck}
  We will apply the results of this subsection mainly to the 
  category
  $\Sh(X)$ of sheaves on a topological space $X$.
  This category is a Grothendieck (abelian) category by
  \cite[18.1.6.(v)]{KS-cat-sh}. We do not need the full strength
  of Proposition~\ref{p:model-structure-C-A} below; the
  results of \cite{KS-cat-sh} are sufficient for our purposes,
  cf.\ \ref{rem:compare-Beke-KS-cat-sh}.
\end{Bemerkungl}

\begin{Proposition}
  [{\cite[Prop.~3.13]{beke-sheafifiable-homotopy-model-categories}}]
  \label{p:model-structure-C-A}
  Let $\mathcal{A}$ be a Grothendieck category. Then
  there is a 
  cofibrantly generated model structure on
  $C(\mathcal{A})$
  such that the weak equivalences are the
  quasi-isomorphisms and the cofibrations are the monomorphisms.
  We call it the \textbf{injective} model structure on
  $C(\mathcal{A}).$ 
\end{Proposition}

\begin{Bemerkungl}
  \label{rem:compare-Beke-KS-cat-sh}
  The fibrant objects of this model structure on
  $C(\mathcal{A})$ are precisely the objects that are
  h-injective and componentwise injective (\cite[Prop.~14.1.6 in
  the setting of section~14.3]{KS-cat-sh} where the class of
  trivial cofibrations is called $\op{QM}$ and the fibrant objects
  are called $\op{QM}$-injective).
  Any bounded below complex of injective objects is h-injective
  (\cite[13.2.4]{KS-cat-sh})
  and hence fibrant. 
  The axioms of a model structure imply that
  any object 
  of $C(\mathcal{A})$
  admits a trivial cofibration to a
  fibrant object. This is also proved in
  \cite[Thm.~14.1.7]{KS-cat-sh}.
\end{Bemerkungl}


\begin{Lemma}
  \label{l:compute-right-derived}
  Let $\mathcal{A}$ be a Grothendieck category and $F
  \colon 
  \mathcal{A} \ra \mathcal{B}$ a left exact functor to an abelian
  category $\mathcal{B}$.
  We also denote by $F\colon K(\mathcal{A}) \ra
  D(\mathcal{B})$ the 
  composition of the induced functor 
  $K(\mathcal{A}) \ra K(\mathcal{B})$
  on homotopy categories with the canonical functor
  $K(\mathcal{B}) \ra D(\mathcal{B}).$
  \begin{enumerate}
  \item
    \label{enum:RF-exists}
    $F\colon K(\mathcal{A}) \ra
    D(\mathcal{B})$ admits a right derived functor $\op{R}\!F\colon
    D(\mathcal{A}) \ra D(\mathcal{B}).$ 
  \end{enumerate}
  \begin{enumerate}[resume]
  \item 
    \label{enum:RF-for-fin-cohodim}
    If $F\colon \mathcal{A} \ra
    \mathcal{B}$ is of finite 
    cohomological dimension,
    then for all 
    componentwise $F$-acyclic complexes
    $S \in C(\mathcal{A})$
    the obvious morphism $F(S) \ra
    \op{R}\! F(S)$ in $D(\mathcal{B})$ is an isomorphism.
  \end{enumerate}
\end{Lemma}

\begin{proof}
  We equip 
  $C(\mathcal{A})$ with the injective model
  structure from
  Proposition~\ref{p:model-structure-C-A}.
  As mentioned in \ref{rem:compare-Beke-KS-cat-sh},
  its fibrant objects are h-injective so that we
  get~\ref{enum:RF-exists}. 
  
  \ref{enum:RF-for-fin-cohodim}:
  Let $\mathcal{S} \subset C(\mathcal{A})$ be the full
  subcategory of componentwise $F$-acyclic complexes. 
  Let $T \in \mathcal{S}$ be acyclic. We claim that $F(T)=0$ in
  $D(\mathcal{B}).$ 
  Let $K^i$ be the kernel of $T^i \ra T^{i+1}.$ The short exact
  sequence $K^{i-1} \hra T^{i-1} \sra K^i$ provides isomorphisms
  $H^p(\op{R}\! F(K^i)) \sira H^{p+1}(\op{R}\! F(K^{i-1}))$ for all
  $p >0.$  
  By induction we obtain 
  $H^p(\op{R}\! F(K^i)) \sira H^{p+N}(\op{R}\! F(K^{i-N}))$ for all
  $p >0$ 
  and all $N \in \DN.$ Since $F$ has finite cohomological
  dimension this shows that $K^i$ is $F$-acyclic.
  Hence we obtain 
  short exact sequences
  $F(K^{i-1}) \hra F(T^{i-1}) \sra F(K^i)$ and see that $F(T)$
  is acyclic.

  Now let $S \ra I$ be a fibrant
  resolution of some object $S \in \mathcal{S}.$ 
  By
  \ref{rem:compare-Beke-KS-cat-sh}
  any fibrant object is componentwise injective thus
  we have $I \in \mathcal{S}.$
  Hence the cone of $S \ra I$ is an acyclic object of
  $\mathcal{S}$ and mapped to zero in $D(\mathcal{B})$ by the
  above claim, so that
  $F(S) \ra F(I)$ is a quasi-isomorphism.
\end{proof}


\def\cprime{$'$} \def\cprime{$'$} \def\cprime{$'$} \def\cprime{$'$}
  \def\Dbar{\leavevmode\lower.6ex\hbox to 0pt{\hskip-.23ex \accent"16\hss}D}
  \def\cprime{$'$} \def\cprime{$'$}
\providecommand{\bysame}{\leavevmode\hbox to3em{\hrulefill}\thinspace}
\providecommand{\MR}{\relax\ifhmode\unskip\space\fi MR }
\providecommand{\MRhref}[2]{%
  \href{http://www.ams.org/mathscinet-getitem?mr=#1}{#2}
}
\providecommand{\href}[2]{#2}

\end{document}